\documentclass{amsart}
\usepackage[ruled, vlined]{algorithm2e}
\usepackage{graphicx}
\usepackage{amsfonts,amssymb,amsmath}
\usepackage{xcolor}
\usepackage[colorlinks=true]{hyperref}
\usepackage{tikz}
\usepackage{mathtools}
\usepackage{color,soul}
\usepackage{setspace}
\usepackage{pgfplots}
\pgfplotsset{compat=newest}
\hypersetup{urlcolor=blue, citecolor=red}
\usepackage{float}
\usepackage{caption}
\usepackage{subcaption}

\newcounter{results}[section]

\newtheorem{theorem}[results]{Theorem}
\newtheorem{definition}[results]{Definition}
\newtheorem{corollary}[results]{Corollary}
\newtheorem{lemma}[results]{Lemma}
\newtheorem{proposition}[results]{Proposition}

\newtheorem{assumption}[results]{Assumption}

\title[Towards Stratified Space Learning: Linearly Embedded Graphs]{Towards Stratified Space Learning: Linearly Embedded Graphs}

\author[Yossi Bokor and Katharine Turner and Christopher Williams]{}

 \keywords{Stratified Space Learning \and Reconstruction and Modelling \and Embedded Spaces \and Implementation}

 \email{yossi.bokor@anu.edu.au}
 \email{katharine.turner@anu.edu.au}
 \email{christopher.williams@anu.edu.au}

\thanks{$^*$ Corresponding author: Yossi Bokor}

\begin{document}
\maketitle

\centerline{\scshape Yossi Bokor$^*$}
\medskip
{\footnotesize
 \centerline{Mathematical Sciences Institute \hspace{0.5cm} School of Mathematics and Statistics} 
   \centerline{Australian National University \,\hspace{1.5cm} The University of Sydney}
   \centerline{Acton, ACT, 2601, Australia \hspace{0.75cm}  Camperdown, NSW, 2006, Australia}
} 
   
   \medskip

\centerline{\scshape Katharine Turner}
\medskip
{\footnotesize
\centerline{Mathematical Sciences Institute}
   \centerline{Australian National University}
   \centerline{Acton, ACT, 2601, Australia}
   }
   
\medskip

\centerline{\scshape Christopher Williams}
\medskip
{\footnotesize
\centerline{Mathematical Sciences Institute}
   \centerline{Australian National University}
   \centerline{Acton, ACT, 2601, Australia}
   }

\bigskip

\begin{abstract}
    In this paper, we consider the simplest class of stratified spaces -- linearly embedded graphs. We present an algorithm that learns the abstract structure of an embedded graph and models the specific embedding from a point cloud sampled from it. We use tools and inspiration from computational geometry, algebraic topology, and topological data analysis and prove the correctness of the identified abstract structure under assumptions on the embedding.  The algorithm is implemented in the  Julia package \href{http://github.com/yossibokor/Skyler.jl}{Skyler}, which we used for the numerical simulations in this paper.
\end{abstract}

\section{Introduction}\label{Sec:intro}
	Increases in the quantity and complexity of collectable data have lead to the search for new methods for efficiently discovering and modelling their underlying structures. The importance of triage and dimensionality reduction of large amounts of data grows with the embedding dimension. By expanding the class of underlying structures which can be detected and modelled, we aim to address some of the difficulties. To improve dimensionality reduction's efficiency and accuracy, we remove the manifold assumption where the dimension is constant and instead treat it as a stratified space, learning the local dimension in the algorithm. We focus on one-dimensional stratified spaces (i.e. graphs) and here provide a new method for dimensionality reduction and compression. 

	Manifold learning is a method of detecting and modelling structures underlying data sets. There are numerous algorithms and theorems for learning geometric and topological features of manifolds from (noisy) samples, such as dimension or the manifold itself (see \cite{deymanifoldreconstruction}, \cite{deycurveandsurface}, \cite{deydimensiondetection}). These algorithms make assumptions about the manifold and the sampling procedure, often in the form of curvature restrictions and conditions on the sample's density and noise. Unfortunately, these assumptions are not satisfied by point clouds arising in many applications, such as geospatial transportation network data of vehicle movement. We move towards resolving this problem by expanding the set of allowable underlying structures to include stratified spaces. A \emph{stratified space} is a space described by gluing together (manifold) pieces, called strata. There are no restrictions placed upon each stratum's dimension, and the gluing can give rise to a variety of interesting and complex local structures. 
	
	Bendich et al. (\cite{stratlearning}, and \cite{localhomology}) describe an algorithm which, under certain conditions, can identify if two points have been sampled from the same stratum of a stratified space. This algorithm does not provide a method for learning the global abstract structure. In related work, Nanda et al. (\cite{nandaintersection}) present an algorithm for detecting when points have been sampled from two intersecting manifolds which is a cruder splitting than the splitting into stratified subspaces. They have some experimental verification but no theoretical guarantees.
	
	The closest previous work to this paper is  \cite{metricgraph}, in which Aanjaneya et al. consider reconstructing \emph{metric graphs} to detect branch points and the graph structure. There are a few crucial differences. They focus in on the reconstruction of the metric, with input intrinsic distances on the metric graph (plus noise) and the aim to reconstruct a metric graph that is homeomorphic and close as metrics. This means that the theoretical guarantees are about the lengths of edges in the metric graph instead of geometric conditions on an embedding. Crucially, they do not need to consider vertices of degree $2$ as in a metric space setting these are points on an edge.
	
	In contrast, this paper describes an algorithm for modelling a linear embedding of a simple graph from a point cloud sample and provide theoretical guarantees in terms of the geometric embedding that the graph structure modelled is equivalent to the structure embedded. 
	
	\begin{definition}[Graph]
		A \emph{graph} $G$ consists of 
		\begin{enumerate}
			\item a set of vertices $V= \{ v_i\}_{i=1}^{n_v}$,
			\item a set of edges $E= \{ (v_{j_{1}}, v_{j_{2}})\}_{j=1}^{n_e}$. 
		\end{enumerate}
		
		For any graph $G$, the \emph{boundary operator} $\partial_G: E \to V \times V$, maps an edge to the two boundary vertices. We can represent $\partial_G$ via the \emph{boundary matrix} $B$, which is the $n_v \times n_e$ matrix with $B[i,j]=1$ if $v_i = v_{j_{1}}$ or $v_i = v_{j_{2}}$. Edges $(v_{j_{1}}, v_{j_{2}})$ are \emph{open}, and their boundary consists of the two vertices.
		\end{definition}

	Given a graph $G$, we can embed it into $\mathbb{R}^n$ in numerous ways. We will restrict to linear embeddings, such that at degree 2 vertices, the angle between edges is not $\pi$.

\begin{definition}[Linear embedding]
		A linearly embedded graph 
		
		\begin{equation*}
		    |G| = (G, \phi_G) \subset \mathbb{R}^n
		\end{equation*} 
		
		is a graph $G$, and a map $\phi_G:  G \to \mathbb{R}^n$, such that 
		\begin{enumerate}
	    	\item on the vertex set $V$, $\phi_G$ is injective, and we denote $\phi_G(v)$ by $v$,
	    	\item on $E$, $\phi_G$ is defined by linear interpolation: the embedding of an edge $(u,v)$ is the line segment joining $\phi_G(u)$ and $\phi_G(v)$, denoted $\overline{\phi_G(u)\phi_G(v)}= \overline{uv}$,
	    	\item embedded edges $\overline{uv}, \overline{u'v'}$ only intersect if they share a boundary vertex, say $v'=v$, and their intersection is $\phi_G(v)$. 
		\end{enumerate}
		We restrict our attention to embedded graphs $|G|$ such that at a degree two vertex $v$, the embedded edges, say $\overline{uv}, \, \overline{wv}$ form an angle $\alpha \neq \pi$.
	\end{definition}

   
   Please note that with an abuse of notation we will usually use $v$ to denote both the abstract vertex and the embedded location $\phi_G(v)$, and use $\overline{uv}$ to denote both the abstract edge and the embedded image of that edge by $\phi_G$. It should always be clear from context whether we are referring to an element in the abstract structure or to its image in $\mathbb{R}^n$.
   
    Throughout this paper, we use the following conventions. For two points $x, y \in \mathbb{R}^n$, $\| x-y \|$ is the distance between $x$ and $y$ in the standard Euclidean metric on $\mathbb{R}^n$, $\langle x, y \rangle$ is the inner product of $x$ and $y$. For a point $x \in \mathbb{R}^n$ and a set $Y \subset \mathbb{R}^n$, we set 
    
    \begin{equation*}
        d(x, Y) := \min_{y \in Y}\| x -y \|,
    \end{equation*} 
    
    and for two sets $X, Y \subset \mathbb{R}^n$, 
    
    \begin{equation*}
        d(X,Y) := \min_{x \in X, y \in Y}\| x-y\|.
    \end{equation*}

	Given a point $x \in |G|$, we can determine if $x$ is on an edge, or is a vertex by considering the intersection of $|G|$ with a small ball around $x$. Consider $B_r(x)$ for small $r >0$. If $x$ is a vertex, $r$ is less than $\|x-w\|$ for all vertices $w \neq x$ and there are no edges $\overline{uw}$ within $r$ of $x$, then $B_r(x) \cap |G|$ is connected, and for each edge containing $x$, there is a unique point in $\partial B_r(x)$.  If $x$ is a degree $2$ vertex, let the two points on $\partial B_r(x)$ be $p$ and $q$, then $\angle pxq < \pi$. Now consider $x \in \overline{uv}$ for some embedded edge $\overline{uv}$, and take $r < \text{min} \left\{\|x-v\|, \|x-u\| \right \}$. If there is some edge $\overline{wz}$ with $d(x, \overline{wz}) \leq r$, then $B_r(x) \cap |G|$ is disconnected. Otherwise, $B_r(x) \cap |G|$ is connected, and there are two points $q,p$ in $\partial B_r(x) \cap |G|$, and $\angle pxw = \pi$. This is an adaption of the local homology of $|G|$ at $x$.

	We suppose that we do not have the entire embedded graph $|G|$, but only a finite sample $P$. Furthermore, we expect noise so that $P \subsetneq |G|$, and we can only make statements about the distance between $P$ and $|G|$. We restrict to sufficiently dense samples $P$ of $|G|$ with bounded noise. Let $d_H(X,Y)$ be the Hausdorff distance between two subsets $X,Y$ of $\mathbb{R}^n$. We consider \emph{$\varepsilon$-samples} of embedded graphs $|G|$.
	
	\begin{definition}[$\varepsilon$-sample]
		Let $|G| \subset \mathbb{R}^n$ be an embedded graph. An \emph{$\varepsilon$-sample $P$} of $|G|$ is a finite subset of $\mathbb{R}^n$ such that $d_H(|G|, P) \leq \varepsilon$.
	\end{definition}  
	
	We can now state the aim of this paper: given an $\varepsilon$-sample $P$ of a linearly embedded graph $|G|$, we want to 1) detect the graph structure $G$, and then 2) model $\phi_G$. This is a semi-parametric problem: the parameters we need to learn are the number of vertices, the number of edges, and the boundary operator. To do so, we we need to decide if $p$ is near a vertex $v$ or far away from all vertices for each $p \in P$. This partitions our sample $P$ into two subsets, which intuitively are $P_0$ containing samples $p$ which are near a vertex, and $P_1$ containing samples $p$ which are not near any vertex. We define $P_0$ and $P_1$ rigorously in Definition \ref{defn:p0p1}. In the process of partitioning $P$, we approximate the previous modification of local homology for each $p \in P$. This requires choosing a scale for our local neighbourhoods, and a scale for approximating $|G|$ from $P$. The clusters in $P_0$ and $P_1$ correspond to vertices and edges in $G$ respectively, and we can use the minimal distance between clusters in $P_1$ and $P_0$ to learn the boundary operator. Using this information, we model the embedding $\phi_G$.
	
	The previous modification of local homology first used the connectedness of $B_r(x) \cap |G|$ to determine if $x$ was not a vertex, and then counted the points in $\partial B_r(x) \cap |G|$ and used their relative geometry to decide if $x$ was a vertex.  Given a a point $p$ near $|G|$, we can consider $B_r(p) \cap |G|$ and $\partial B_r(x) \cap |G|$ to determine if $p$ is near a vertex or not. As $p$ is within $\varepsilon$ of $|G|$, $r$ must be greater than $\varepsilon$ to ensure $B_r(x) \cap |G|$ is non-empty.
	
	Fix $R >\varepsilon$. We first want to approximate $B_R(x) \cap |G|$, and then $\partial B_R(x) \cap |G|$ from $P$. We can approximate $B_R(p) \cap |G|$ by considering samples $q \in P$ with $\|p-q\| \leq R$. As $P$ is an $\varepsilon$-sample of $|G|$, we can approximate $\partial B_R(p) \cap |G|$ by considering the samples in a spherical shell $S_{R-\varepsilon}^{R+\varepsilon}(p)$ of inner radius $R- \varepsilon$, outer radius $R+\varepsilon$ around $p$. 

	
	We model $\phi_G$ by aiming to reconstruct a probability measure $\nu$ which is supported on $|G| \subset \mathbb{R}^n$.
	As recorded data has errors, we cannot directly reconstruct $\nu$, but instead construct an approximating measure $\nu_{\delta}$ such that $\nu_{\delta}$ is equivalent to the Lebesgue measure, and $\text{supp}(\text{lim}_{\delta\rightarrow 0}\nu_{\delta}) = |G|$.
	We form $\nu_{\delta}$ from a categorical mixture model of measures over the individual strata pieces, with latent variables for strata assignment. 
	We use a Gaussian convolution for each individual strata piece to form our approximation of $\nu$ with $\nu_\delta$.
	We derive a log-likelihood function which is maximised through an Expectation-Maximisation algorithm (Algorothm \ref{alg:vertex}). 
	
	In Section \ref{sec:geolems}, we present and prove some geometric lemmas used throughout Sections \ref{sec:localstruct} and \ref{sec:alg}, then in Section \ref{sec:localstruct} we define $(R,\varepsilon)$-local structure, describe the $(R,\varepsilon)$-local structure of a vertex and of an edge, before providing conditions under which we can guarantee what $(R,\varepsilon)$-local structure a sample $p$ has. Section \ref{sec:alg} presents the algorithm, relates it to the $(R,\varepsilon)$-local structure, before proving that the abstract graph identified is equivalent to the original one. Finally, Section \ref{sec:modelling} describes the modelling process used, and contains some simulations. 

    
\section{Some Geometric Lemmas}\label{sec:geolems}
    As motivation for the formulas both in the definitions of local structure and the geometric assumptions of the graphs' embedding, we first prove some geometric lemmas. Throughout our process, we consider $\langle x_1-p, x_2-p\rangle$ for $p,x_1,x_2$ samples, and $\|p - x_1\|, \|p-x_2\| \in [R-\varepsilon, R+ \varepsilon]$. In particular, if there are two clusters of points in the spherical shell around a sample $p$, all points (including $p$) are within $\varepsilon$ of an edge $\overline{uv}$, and $x_1$ and $x_2$ are from different clusters, we wish to bound $\langle x_1-p, x_2-p\rangle$ from above. 

	\begin{lemma}\label{lem:getedges}
	    Fix $R > 12 \varepsilon >0$ and consider a sample $p$ within $\varepsilon$ of an edge $\overline{uv}$. Let $H$ be the hyper-plane through $p$ perpendicular to $\overline{uv}$. Now take $ x_1, x_2$ within $\varepsilon$ of edge $\overline{uv}$ such that $\|p-x_1\|, \|p-x_2\|\in [R-\varepsilon, R+\varepsilon]$ and $x_1, x_2$ are on different sides of $H$. Then
	    
	    \begin{equation*}
	           \langle x_1-p, x_2-p\rangle \leq -R^2+2R\varepsilon+7\varepsilon^2.
        \end{equation*}
	\end{lemma}
	
	\begin{proof}
	    By assumption $\|x_1-p\|, \|x_2-p\|\geq R-\epsilon$. As $x_1, p, x_2$ are all within $\varepsilon$ of $\overline{uv}$ we know that $\angle(x_1 p x_2)\in [\pi-2\arccos(\frac{2\epsilon}{R-\epsilon}), \pi]$. Together we can bound 
	
	    \begin{align*}
	        \langle x_1-p, x_2-p\rangle&=\|x_1-p\|\|x_2-p\|\cos \angle(x_1 p x_2)\\
	        &\leq (R-\epsilon)^2\cos \left(\pi-2\arccos\left(\frac{2\epsilon}{R-\epsilon}\right)\right)\\
	        &\leq (R-\epsilon)^2\left(2\frac{(2\epsilon)^2}{(R-\epsilon)^2}-1\right)\\
	        &\leq -R^2+2R\varepsilon+7\varepsilon^2.
	    \end{align*}
	\end{proof}
  	
  	We want to distinguish points very close to a vertex of degree $2$ as close to a vertex, from points on an edge. This requires an upper bound on the angle at any vertex of degree $2$ within our geometric assumptions due to the noise in sampling. The following geometric lemma motivates the upper bound given in the next section.  
  		
	\begin{lemma}\label{lem:anglebound}
		 Fix $R \geq 12 \varepsilon >0$. Take $u,v,w \in \mathbb{R}^n$, and consider the line segments $\overline{uv}, \overline{wv}$. 
		
		Let $p,x_1, x_2 \in \mathbb{R}^n$ be such that $p$ and $x_1$ are within $\varepsilon$ of $\overline{vw}$, $x_2$ is within $\varepsilon$ of $\overline{uv}$, and $\|x_1-p\|, \|x_2-p\|\in [R-\varepsilon, R+\varepsilon]$.

		If either
		\begin{enumerate}
		    \item $\|p-v\|<4\varepsilon$ and 	
		    \begin{equation*}
		        \pi/2< \angle uvw < \pi-\arctan\left(\frac{R+3\varepsilon}{6\varepsilon}\right) +\arcsin\left(\frac{R^2-4R\varepsilon-9\varepsilon^2}{(R+\varepsilon)\sqrt{R^2+6R\varepsilon+34\varepsilon^2}}\right),
		    \end{equation*}
		    
		    OR
		    \vspace{0.3cm}
		    \item $\|p-v\|<(R-\varepsilon)/2$ and $\angle uvw\leq \pi/2$
		\end{enumerate}
		
	then	
		\begin{equation*}
		    \langle x_1-p, x_2-p\rangle > -R^2 +2R\varepsilon+7\varepsilon^2.
		\end{equation*}
    \end{lemma}
	
	\begin{proof}
		Let $\widetilde{p}, \widetilde{x_1}, \widetilde{x_2}$ be the projections of $p, x_1, x_2$ to $\overline{uv}\cup \overline{wv}$. Without loss of generality, we assume $\widetilde{p}, \widetilde{x_1} \in \overline{wv} \cup v$, and $\widetilde{x_2} \in \overline{uv}$. Then there are $e_p, e_1, e_2 \in \mathbb{R}^n$ with $\|e_q\|, \|e_1\|, \|e_2\| \leq \varepsilon$ and 
		\begin{align*}
			p   &= \widetilde{p}+ e_p\\
			x_1 &= \widetilde{x_1}+e_1\\
			x_2 &= \widetilde{x_2}+e_2.
		\end{align*} 
    
        Now consider the vectors $x_1 -p$ and $x_2-p$, we have:
		\begin{equation}\label{eq:innerprod}
			\langle x_1 -p, x_2 -p \rangle
			= \langle \widetilde{x_1} -\widetilde{p},   \widetilde{x_2} -\widetilde{p}\rangle + \langle \widetilde{x_1} -\widetilde{p}, e_2 \rangle -\langle \widetilde{x_1} -\widetilde{p}, e_p \rangle +\langle e_1 - e_p, x_2 -p\rangle
		\end{equation}
		
	    We know that $e_p$ is perpendicular to $\overline{vw}$ and thus it is also perpendicular to $\widetilde{x}_1-\widetilde{p}$ implying $\langle \widetilde{x_1} -\widetilde{p}, e_p \rangle=0$. 
	
	    We also know that $\|\widetilde{x}_1-\widetilde{p}\|\leq \|x_1-p\|\leq R+\varepsilon$ as distances can only decrease when projecting onto $\overline{vw}$.
		
	    To bound $\langle \widetilde{x_1} -\widetilde{p}, e_2 \rangle$ we first split $e_2=e_2' +e_2''$ where $e_2'$ is the projection of $e_2$ into the plane spanned by $\overline{vu}$ and $\overline{vw}$. Note that $e_2''$ is perpendicular to $\widetilde{x_1} -\widetilde{p}$ and hence $\langle \widetilde{x_1} -\widetilde{p}, e_2 \rangle=\langle \widetilde{x_1} -\widetilde{p}, e_2' \rangle$.
		
        From here, we need to split the proof into the two scenarios.
	
        Assume we are in scenario 1: $\|p-v\|<4\varepsilon$ and
    
        \begin{equation*}
            \pi/2< \angle uvw < \pi-\arctan\left(\frac{R+3\varepsilon}{6\varepsilon}\right) +\arcsin\left(\frac{R^2-4R\varepsilon-9\varepsilon^2}{(R+\varepsilon)\sqrt{R^2+6R\varepsilon+34\varepsilon^2}}\right).
        \end{equation*}
        
        The angle between $e_2'$ and $\widetilde{x_1} -\widetilde{p}$ is either $\angle uvw+\pi/2$ or $\angle uvw -\pi/2$. Recall that we assumed $\angle uvw\in(\pi/2,\pi)$, so $\cos(\angle uvw-\pi/2)>0>\cos(\angle uvw+\pi/2)$ and
		\begin{equation}\label{eq:e2'}
		   \langle \widetilde{x_1} -\widetilde{p}, e_2 \rangle
		   = \langle \widetilde{x_1} -\widetilde{p}, e_2' \rangle
		   \geq \|\widetilde{x_1}-\widetilde{p}\|\|e_2'\|\cos(\angle uvw+\pi/2)
		   \geq -\varepsilon(R+\varepsilon)\sin \angle uvw.
		\end{equation}
		
		Combining \eqref{eq:innerprod} and \eqref{eq:e2'} we see
		\begin{equation}\label{eq:innerprod2}
		    \langle x_1 -p, x_2 -p \rangle
			\geq \langle \widetilde{x_1} -\widetilde{p}, \widetilde{x_2} -\widetilde{p}\rangle -\sin \angle uvw (R+\varepsilon)\varepsilon -(R+\varepsilon)2\varepsilon.
		\end{equation}

	    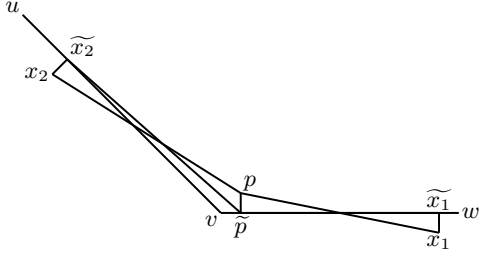
\begin{figure}[H]
		    \centering
		    \tikzset{every picture/.style={line width=0.75pt}} 

		    \begin{tikzpicture}[x=0.75pt,y=0.75pt,yscale=-1,xscale=1]
			    \draw (0,0) -- (120,0);
			    \draw (0,0) -- (-100,-100);
			    \draw (120,0) node [font=\small][anchor=west] [inner sep =0.75pt] {$w$};
			    \draw (-100,-100) node [font=\small][anchor=south east] [inner sep =0.75pt] {$u$};
			    \draw (0,0) node [font=\small][anchor=north east] [inner sep =0.75pt] {$v$};
			    \draw (10,-10) -- (-85,-70);
			    \draw (10,-10) -- (110,10);
			    \draw (10,-10) node [font=\small][anchor=south west] [inner sep =0.75pt] {$p$};
			    \draw (-85,-70) node [font=\small][anchor=east] [inner sep =0.75pt] {$x_2$};
			    \draw (110,10) node [font=\small][anchor=north] [inner sep =0.75pt] {$x_1$};
			    \draw (-85,-70) -- (-77.5, -77.5);
			    \draw (-77.5,-77.5) node [font=\small][anchor=south west] [inner sep =0.75pt] {$\widetilde{x_2}$};
			    \draw(110,10) -- (110,0);
			    \draw (110,0) node [font=\small][anchor = south][inner sep =0.75] {$\widetilde{x_1}$};
			    \draw (10,-10) -- (10,0);
			    \draw (10,0) node [font=\small][anchor = north][inner sep =0.75pt] {$\widetilde{p}$};
			    \draw (10,0) -- (-77.5,-77.5);
			    \draw (10,0) -- (110,0);
		    \end{tikzpicture}
		    \caption{An example of scenario 1.}\label{fig:angleub}
		\end{figure}
		
		To bound $\langle \widetilde{x_1}-\widetilde{p}, \widetilde{x_2}-\widetilde{p}\rangle $ we use that $\angle \widetilde{x_1} \widetilde{p} \widetilde{x_2}= \angle uvw+\angle  v \widetilde{x_2} \widetilde{p}$, that the sine rule says $\|\widetilde{x_2}-\widetilde{p}\|\sin (\angle \widetilde{x_2}v\widetilde{p})=\|v-\widetilde{p}\|\sin \angle uvw$, and that $\cos \angle v\widetilde{x_2}\widetilde{p}>0$, $\cos\angle uvw<0$ and $-\sin^2\angle uvw\leq-\sin\angle uvw$. Together these imply that
		
		\begin{align*}
		    \langle \widetilde{x_1}-\widetilde{p}, \widetilde{x_2}-\widetilde{p}\rangle  &=\|\widetilde{x_1}-\widetilde{p}\|\|\widetilde{x_2}-\widetilde{p}\|\cos(\angle uvw+ \angle  v \widetilde{x_2} \widetilde{p} )\\
		    &=\|\widetilde{x_1}-\widetilde{p}\|\|\widetilde{x_2}-\widetilde{p}\|\cos\angle uvw\cos(\angle  v \widetilde{x_2} \widetilde{p})- \sin^2\angle uvw\|v-\widetilde{p}\|\|\widetilde{x_1}-\widetilde{p}\|\\
		    &\geq (R+\varepsilon)(R+3\varepsilon)\cos\angle uvw-\sin\angle uvw\|v-\widetilde{p}\|(R+\varepsilon).
		\end{align*}
	
        From the assumptions in this scenario that $\|v-\widetilde{p}\|\leq 4\varepsilon$, we can substitute into \eqref{eq:innerprod2} to get
    
	    \begin{align*}
	        &\langle x_1 -p, x_2 -p \rangle \\
	        &\geq (R+\varepsilon)(R+3\varepsilon) \cos \angle uvw - 4\varepsilon(R+\varepsilon) \sin \angle uvw - R\varepsilon(2 + \sin \angle uvw ) - (2 + \sin \angle uvw) \varepsilon^2\\
    	    &=(R+\varepsilon)\sqrt{R^2+6R\varepsilon+34\varepsilon^2}\sin\left(\angle uvw + \arctan\left(\frac{R+3\varepsilon}{5\varepsilon}\right)\right)-2\varepsilon R-2\varepsilon^2.
	    \end{align*} 
	
        From our assumptions in $\angle uvw$ 
    
    
   

        \begin{align*}
            \sin\left(\angle uvw +\arctan\left(\frac{R+3\varepsilon}{5\varepsilon}\right)\right) &>-\frac{R^2-4R\varepsilon+\varepsilon^2}{(R+\varepsilon)\sqrt{R^2+6R\varepsilon+34\varepsilon^2}}.\\
        \end{align*}

        Thus we conclude
    
        \begin{align*}
            \langle &x_1 -p, x_2 -p \rangle\\
		    &>(R+\varepsilon)\sqrt{R^2+6R\varepsilon+34\varepsilon^2}\left(-\frac{R^2-4R\varepsilon-9\varepsilon^2}{(R+\varepsilon)\sqrt{R^2+6R\varepsilon+34\varepsilon^2}}\right) -2\varepsilon R-2\varepsilon^2\\
            &=-R^2+2R\varepsilon+7\varepsilon^2.
        \end{align*}

        Now assume we are in scenario 2: $\|v-p\|<(R-\varepsilon)/2$ and $\angle u v w\leq \pi/2$.
        
        To prove this scenario, we will need to further split into two cases;
    
        \begin{enumerate}
            \item[(i)] $\angle \widetilde{x_1} \widetilde{p} \widetilde{x_2}\leq \pi/2$, and
            \item[(ii)] $\angle \widetilde{x_1} \widetilde{p} \widetilde{x_2} >\pi/2$. 
        \end{enumerate}

        In case (i) we have  $\langle\widetilde{x_1}-\widetilde{p}, \widetilde{x_2}-\widetilde{p}\rangle  \geq 0$ and thus  
        
        \begin{equation*}
            \langle x_1-p, x_2-p\rangle \geq - 3R\varepsilon-3\varepsilon^3>-R^2+2R\varepsilon +7\varepsilon^2
        \end{equation*}
        as $R>12\varepsilon$.
    
        In case (ii), thinking of the inner product in terms of the projection of vector $\widetilde{x_2}-\widetilde{p}$ onto $\widetilde{x_1}-\widetilde{p}$ we get

        \begin{align*}
           \langle x_1-p, x_2-p\rangle &\geq  -\|\widetilde{x_1}-\widetilde{p}\|\|v-\widetilde{p}\| - 3R\varepsilon-3\varepsilon^3\\
           &\geq -(R+\varepsilon)(R-\varepsilon)/2 -3R\varepsilon-3\varepsilon^3\\
           &=-R^2/2-3R\varepsilon-5\varepsilon^2/2\\
           &>-R^2+2R\varepsilon+7\varepsilon^2,
        \end{align*}
  
        where in the final inequality we use that $R>12\varepsilon$.
    \end{proof}
    
    To find sufficient conditions for detecting when a sample $p$ is near a vertex, we want each edge adjacent to that vertex to correspond to at least one distinct cluster of points in the spherical shell around $p$. To avoid the clusters around separate edges merging, we assume a lower bound on the angle between the edges as part of our assumptions on the geometric embedding. The following lemma motivates this choice of lower bound. 
  	
  	\begin{lemma}\label{lem:sepedges}
  	    Let $u,v,w\in \mathbb{R}^n$, $D > \varepsilon > 0$, and let $x_1, x_2\in \mathbb{R}^n$ satisfy
  	    
        \begin{enumerate} 
      	    \item $d(x_1, \overline{uv}),d(x_2, \overline{uw})<\varepsilon$, and
      	    \item $\|x_1 -v\|,\|x_2- v\|>D$.
        \end{enumerate}

        If 
  	    
  	    \begin{equation*}
  	        \angle uvw>\arccos\left(\frac{2D^2-9\varepsilon^2}{2D^2}\right)+2\arcsin\left(\frac{\varepsilon}{D}\right)
  	    \end{equation*}
  	
        then $\|x_1-x_2\|> 3\varepsilon$.
	\end{lemma}
	
	\begin{proof}

	    The distance between $x_1$ and $x_2$ is  minimised when $\|v - x_1\|= D= \|v-x_2\|$. Furthermore we can observe that $\angle u v x_1=\arcsin\left(\frac{d(x_1,\overline{uv})}{\|x_1- v\|}\right)\leq\arcsin(\varepsilon/D)$. Similarly $\angle u v x_1 \leq\arcsin(\varepsilon/D)$. This implies
	
	    \begin{equation*}
	        \angle x_1 v x_2\geq \angle u v w -\angle u v x_1 -\angle w v x_2 \geq \alpha-2\arcsin(\varepsilon/D).
        \end{equation*}
	
	    Combining we conclude
	
	    \begin{align*}
	        \|x_1- x_2\|^2&\geq \|v-x_1\|^2+\|v-x_2\|^2-\|v-x_1\|\|v-x_2\|\cos\angle x_1 v x_2\\
	        &\geq 2D^2 -2D^2\cos(\alpha-2\arcsin(\varepsilon/D))\\
	        &\geq(3\varepsilon)^2.
	    \end{align*}

	\end{proof}	

\section{Determining Local Structure}\label{sec:localstruct}

	Given an $\varepsilon$-sample $P$ of an embedded graph $|G|$, we want to recover the abstract graph $G$ by approximating the local structure of $|G|$ at each sample $p \in P$. In this process, we regularly consider the graph on a set of points, with edges $(p,q)$ if $\|p-q\| \leq r$, for some fixed $r \in \mathbb{R}$.
	
	\begin{definition}
		Let $P \subset \mathbb{R}^N$ be a finite collection of points, and fix $r> 0$. The \emph{graph at threshold $r$ on $P$}, $\mathfrak{G}_r(P)$, is the graph with vertices $p \in P$, and edges $(p,q)$ if $\|p-q\| \leq r$.
	\end{definition}
	
	For each $p \in P$, we will consider two graphs on points close to $p$: the first approximates the connectedness of $|G|$ intersected with a ball around $p$, the second consists of points in a spherical shell around $p$. We call this pair of graphs the \emph{$(R, \varepsilon)$-local structure of $P$ at $p$}.
	
	\begin{definition}[$(R, \varepsilon)$-local structure]
		Let $P \subset \mathbb{R}^n$ be an $\varepsilon$-sample of an embedded graph $|G|$ and fix $R > 12 \varepsilon$. The $(R, \varepsilon)$-local structure of $P$ at a sample $p\in P$ is the pair 
		
		\begin{equation*}
			\left(\mathfrak{G}_{3\varepsilon}(P \cap B_{R+ \varepsilon}(p)), \mathfrak{G}_{3\varepsilon}(P \cap S_{R-\varepsilon}^{R+\varepsilon}(p))\right).
		\end{equation*}
	\end{definition}
	
	We want to use the $(R, \varepsilon)$-local structure to approximate $|G| \cap B_R(p)$ for each $p \in P$, and use this to learn the structure of $|G|$. We will classify samples as being near a vertex or not near a vertex by their $(R,\varepsilon)$-local structure. 
	
	We now formalise what the $(R, \varepsilon)$-local structure is for points $p \in P$ not near any vertex $v \in |G|$. That is, points which have $(R,\varepsilon)$-local structure of an edge. 
	
	\begin{definition}[Local structure of an edge]
		Let $P$ be an $\varepsilon$-sample of a linearly embedded graph $|G|$. A point $p \in P$ has the \emph{$(R,\varepsilon)$-local structure of an edge} if either of the following hold:
		
		\begin{enumerate}
			\item $\mathfrak{G}_{3\varepsilon}(P \cap B_{R+ \varepsilon}(p))$ is disconnected, 
			\item $\mathfrak{G}_{3\varepsilon}(P \cap B_{R+ \varepsilon}(p))$ is connected, $\mathfrak{G}_{3\varepsilon}(P \cap S_{R-\varepsilon}^{R+\varepsilon}(p))$ has two connected components $c_1, c_2$ with average points $q_1$ and $q_2$, and
			\begin{equation*}
			    \langle q_1-p, q_2-p\rangle \leq -R^2+2R\varepsilon+7\varepsilon^2.
			\end{equation*}
		\end{enumerate}
	
	\end{definition}		

	We now define the \emph{$(R,\varepsilon)$-local structure of a vertex}.
	
	\begin{definition}[Local structure of a vertex]
		Let $P$ be an $\varepsilon$-sample of a linearly embedded graph $|G|$. A point $p \in P$ has the \emph{$(R,\varepsilon)$-local structure of a vertex} if either of the following hold:
		\begin{enumerate}
			\item $\mathfrak{G}_{3\varepsilon}(P \cap B_{R+ \varepsilon}(p))$ is connected, and the number of connected components in $\mathfrak{G}_{3\varepsilon}(P \cap S_{R-\varepsilon}^{R+\varepsilon}(p))$ is not 2, 
			\item $\mathfrak{G}_{3\varepsilon}(P \cap B_{R+ \varepsilon}(p))$ is connected, $\mathfrak{G}_{3\varepsilon}(P \cap S_{R-\varepsilon}^{R+\varepsilon}(p))$ has two connected components $c_1, c_2$ with average points $q_1$ and $q_2$, and 
			\begin{equation*}
			    \langle q_1-p, q_2-p\rangle > -R^2+2R\varepsilon+7\varepsilon^2.
			\end{equation*}
			\color{black}
		\end{enumerate}
	\end{definition}	
	
	Next, we formally define $P_0$ and $P_1$.
	
	\begin{definition}[$P_0$ and $P_1$]\label{defn:p0p1}
	    Given an $\varepsilon$-sample $P$ of a linearly embedded graph $|G| \subset \mathbb{R}^n$, we define the partitioning sets $P_0$ and $P_1$ as follows:
	    \begin{align*}
	        P_0 &= \{p \in P \mid \text{ $p$ has the $(R,\varepsilon)$-local structure of a vertex.} \}\\
	        P_1 &= \{p \in P \mid \text{ $p$ has the $(R,\varepsilon)$-local structure of an edge.} \}\\
	    \end{align*}
	\end{definition}

	As we use the connected components of $\mathfrak{G}_{3\varepsilon}(P \cap B_{R+ \varepsilon}(p))$ and  $\mathfrak{G}_{3\varepsilon}(P \cap S_{R-\varepsilon}^{R+\varepsilon}(p))$ in the definition of the $(R,\varepsilon)$-local structure of $p$, we require some assumptions on $|G|$ to ensure that we correctly identify when points are near vertices or not. To ensure $\mathfrak{G}_{3\varepsilon}(P \cap B_{R+ \varepsilon}(p))$  is not disconnected for points $p$ near some vertex, we assume that the distance between a vertex $v$ and any edge $\overline{uw}$, $u,v \neq v$, is bounded below $d(v, \overline{wv}) > R +\frac{R}{2} + 2 \varepsilon$. To ensure that there are samples near edges which are not near any vertex, we additionally assume that for every pair of vertices $u,v$, $\|u-v\| > \frac{9R}{2}+6\varepsilon$. 
	
	We also place lower and upper bounds on the angles between edges. For ease of notation, we will define two functions for these bounds.

	\begin{definition}\label{defn:angleub}
		Given $R > 12 \varepsilon$, we set 
	    \begin{align*}
		    \Psi(R, \varepsilon) &:=  \pi-\arctan\left(\frac{R+3\varepsilon}{6\varepsilon}\right) +\arcsin\left(\frac{R^2-4R\varepsilon-9\varepsilon^2}{(R+\varepsilon)\sqrt{R^2+6R\varepsilon+34\varepsilon^2}}\right),\\
		    \Phi(R,\varepsilon) &:= \arccos\left(\frac{(R-\varepsilon)^2-18\varepsilon^2}{(R-\varepsilon)^2}\right)+2\arcsin\left(\frac{2\varepsilon}{(R-\varepsilon)}\right).\\
        \end{align*}
	\end{definition} 	
	
	To improve intuition of these functions, Figures \ref{fig:Psi} and \ref{fig:Phi} provide graphs of them. Note they are effectively a function of $\frac{R}{\varepsilon}$ as they are invariant to scaling both $R$ and $\varepsilon$ by the same amount.
	
	\begin{figure}[H]\label{fig:Phi_Psi}
	     \includegraphics[width=300pt]{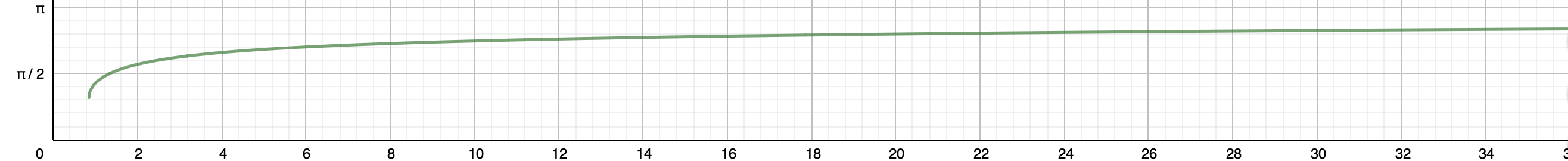}
	     \caption{Graph of $\Psi\left(\frac{R}{\varepsilon}, 1\right)$.}\label{fig:Psi}
	\end{figure}
	
    \begin{figure}[H]
	     \includegraphics[width=300pt]{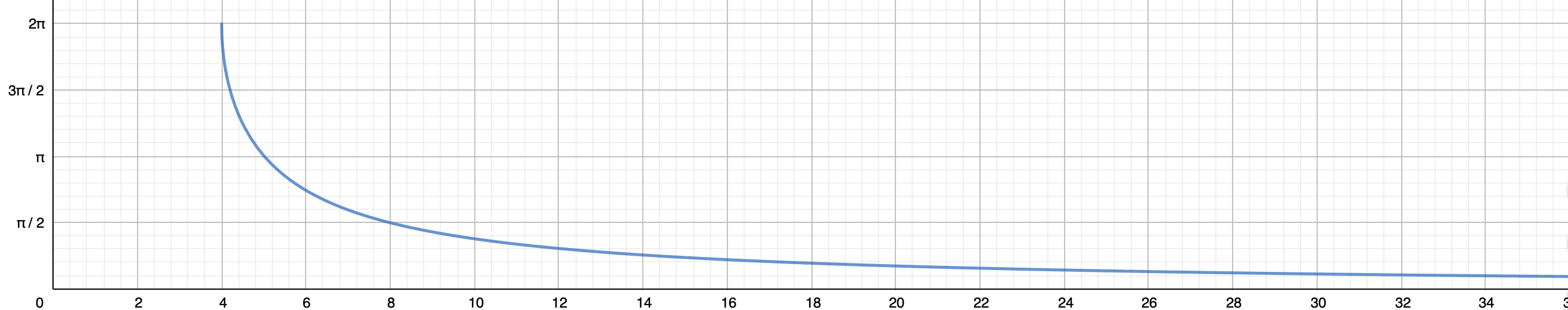}
	      \caption{Graph of $\Phi\left(\frac{R}{\varepsilon}, 1\right)$.}\label{fig:Phi}
	\end{figure}
	
	Henceforth, we assume that all embedded graphs $|G|$ satisfy the following assumptions.
	
	\begin{assumption}\label{as:embed}
		Fix $R\geq 12\varepsilon>0$. We restrict to embedded graphs $|G|= (G, \phi_G)$ satisfying the following.
		\begin{enumerate}
		\item For all vertices $u,v$, $\|u-v\| > \frac{9R}{2} + 6 \varepsilon$. 
		\item For a vertex $v$ and an edge $\overline{uw}$, with $u,w \neq v$, $d(v, \overline{uw}) >  \frac{3R}{2} + 4 \varepsilon$. 
		\item For any pair of edges $\overline{uv}, \overline{xy}$ with no common vertex, $d(\overline{uv}, \overline{xy}) > 5 \varepsilon$.
		\item For all pairs of edges $\overline{uv}, \overline{wv}$, $\angle uvw \geq \Phi(R, \varepsilon)$.
		\item For all degree 2 vertices $v$ with edges $\overline{uv}, \overline{wv}$, $\angle uvw \leq \Psi(R, \varepsilon)$.
	\end{enumerate}
	\end{assumption}

	The propositions in this section are used to show that the clusters in $P_0$ and $P_1$ correspond bijectively with the vertices and edges of $|G|$. The first proposition shows that for all samples $p$ near a vertex $v$ with $\text{deg}(v) \neq 2$, $p$ has the $(R,\varepsilon)$-local structure of a vertex. The second and third prove that samples near degree $2$ vertices also have the $(R,\varepsilon)$-local structure of a vertex. The final proposition shows that all samples $p$ not near any vertex have the $(R,\varepsilon)$-local structure of an edge. 

	\begin{proposition}\label{prop:0ball}
		Let $v$ be a vertex of $|G| \subset \mathbb{R}^n$ with $\text{deg}(v) \neq 2$, and let $P$ be an $\varepsilon$-sample of $|G|$. Then for all  $p \in P$ with $\|p-v\| \leq \frac{R-\varepsilon}{2}$, $p$ has the $(R,\varepsilon)$-local structure of a vertex.
	\end{proposition}

	\begin{proof}
		We begin by considering $\textbf{deg}\mathbf{(v) =0}$. By Assumptions \ref{as:embed} (1), $\|p-v\| \leq \varepsilon$, and for all $q \in P \cap B(p,R+\varepsilon)$, $\|q-v\| \leq \varepsilon$. Thus $\mathfrak{G}_{3\varepsilon}\left( P \cap B_{R+\varepsilon}(p)\right)$ is connected. Similarly, $P \cap S_{R-\varepsilon}^{R+\varepsilon}(p) =\emptyset$, and $p$ has the $(R,\varepsilon)$-local structure of a vertex.
		
		Next, assume $\textbf{deg}\mathbf{(v) = 1}$. For the edge $\overline{uv}$, let $t_0, t_1, \ldots, t_m$ be consecutive points along $\overline{uv}$ with $\| t_0 -v\|, \| t_{i+1} - t_i\| \leq \varepsilon$ and $\|p - t_m\| = R+ \varepsilon$. Then, there must be $z_0, z_1, \ldots, z_m \in P$ with $\|t_i - z_i \| \leq \varepsilon$. Note, these $z_i$ may not be unique. Since $\|z_i- z_{i+1}\| \leq 3 \varepsilon$, and every sample in $P \cap B_{R+\varepsilon}(p)$ is within $3\varepsilon$ of some $z_i$, $\mathfrak{G}_{3\varepsilon}\left( P \cap B_{R+\varepsilon}(p)\right)$ is connected.
	
	    If the number of clusters in $\mathfrak{G}_{3\varepsilon}(P \cap S_{R-\varepsilon}^{R+\varepsilon}(p))$ is not $2$, then $p$ has the $(R, \varepsilon)$-local structure of a vertex. Thus suppose that there are $2$ connected components. We will show that inner product condition between their averages will declare that $p$ has the $(R, \varepsilon)$-local structure of a vertex. 
	
	    Let $x_1, x_2\in P \cap S_{R-\varepsilon}^{R+\varepsilon}(p)$ be samples in the two connected components $c_1$ and $c_2$. Observe that both $x_1$ and $x_2$ are within $\varepsilon$ of the line $\overline{uv}$.
		
		As $\|p-v\| \leq \frac{R-\varepsilon}{2}$, and $x_1, x_2$ are within $\varepsilon$ of the same edge $\overline{uv}$, $x_1$ and $x_2$ are contained on the same side of hyper-plane $H$ through $p$ perpendicular to $\overline{vu}$. 
		
		We can observe that $\angle x_1 p x_2\leq 2\arccos\left(\frac{2\varepsilon}{R-\varepsilon}\right)<\pi/2$, and thus
		
		\begin{align*}
		    \langle x_1 -p, x_2 -p\rangle >0>-R^2+2R\varepsilon+7\varepsilon^2.
		\end{align*}
		
		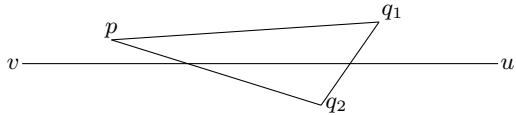
\begin{figure}[H]
		    \centering
		    \begin{tikzpicture}[x=0.75pt,y=0.75pt,yscale=-1.5,xscale=1.5]
			    \draw (-30,0) -- (130,0);
			    \draw (0,-8) node [anchor=south][font=\small][inner sep =0.5pt] {$p$};
			    \draw (90,-14) -- (70.5,14);
			    \draw (-30,0) node [anchor=east][font=\small][inner sep=0.75pt] {$v$};
			    \draw (70.5,14) node [anchor=west][font=\small][inner sep=1.5pt] {$q_2$};
			    \draw (90,-14) node [anchor=south west][font=\small][inner sep=0.75pt] {$q_1$};
			    \draw (0,-8) -- (90,-14);
			    \draw (0,-8) -- (70.5,14);
			    \draw (130,0) node [anchor=west][font=\small][inner sep=0.75pt] {$u$};
		    \end{tikzpicture}
		    \caption{Both $q_1$ and $q_2$ are in the same half-space generated by the hyper-plane through $p$ perpendicular to $\overline{uv}$.}\label{fig:deg1cc}
		\end{figure}
	
	    As this holds for all $x_1 \in c_1, x_2 \in c_2$, it also holds for the averages $q_1$ and $q_2$. Thus $p$ has the $(R,\varepsilon)$-local structure of a vertex.

		Finally, assume $\textbf{deg}\mathbf{(v) \geq 3}$. From analogous arguments as in the degree $1$ case we know that $\mathfrak{G}_{3\varepsilon}\left( P \cap B_{R+\varepsilon}(p)\right)$ is connected.
		
		
		Now consider $S_{R-\varepsilon}^{R+\varepsilon}(p)$. For each edge $\overline{uv}$, there is a sample $x_{\overline{uv}} \in S_{R-\varepsilon}^{R+\varepsilon}(p)$. To show there are at least $3$ connected components in  $\mathfrak{G}_{3\varepsilon}( P \cap S_{R-\varepsilon}^{R+\varepsilon}(p))$, we need only check that samples from different edges cannot merge to be in the same connected component in $\mathfrak{G}_{3\varepsilon}( P \cap S_{R-\varepsilon}^{R+\varepsilon}(p))$. By way of contradiction suppose there were edges $\overline{uv}$ and $\overline{wv}$ and samples $x_u, x_v\in P \cap S_{R-\varepsilon}^{R+\varepsilon}(p)$ within $\varepsilon$ of $\overline{uv}$ and $\overline{wv}$ respectively such that $\|x_u- x_w\|\leq 3\varepsilon$. As $\|p-v\|\leq (R-\varepsilon)/2$ and $\|p-x_u\|,\|p-x_v\|\geq R-\varepsilon$ we know $\|v- x_u\|, \|v- x_w\|>(R-\varepsilon)/2$. This contradicts Lemma \ref{lem:sepedges} as this implies that $\|x_u- x_v\|>3\varepsilon.$
		
	    We conclude that $\mathfrak{G}_{3\varepsilon}( P \cap S_{R-\varepsilon}^{R+\varepsilon}(p))$ has at least as many connected components as the degree of $v$. Thus, $p$ has the $(R,\varepsilon)$-local structure of a vertex.
	\end{proof}
	
	\begin{proposition}\label{prop:0balldeg2big}
		Let $v$ be a vertex of $|G| \subset \mathbb{R}^n$ with $\text{deg}(v) = 2$, with edges $\overline{uv}, \overline{wv}$. Let $P$ be an $\varepsilon$-sample of $|G|$. If $\angle uwv > \frac{\pi}{2}$, then for all  $p \in P$ with $\|p-v\| \leq 4\varepsilon$,
		$p$ has the $(R,\varepsilon)$-local structure of a vertex.
	\end{proposition}
	
	\begin{proof}
		As in the proof Proposition \ref{prop:0ball}, $\mathfrak{G}_{2\varepsilon}(P \cap B_{R+\varepsilon}(p))$ is connected. 
		
		For both edges $\overline{uv}, \overline{wv}$ there is at least one sample in $S_{R-\varepsilon}^{R+\varepsilon}(P)$, say $q_{\overline{uv}}$ and $q_{\overline{wv}}$. By Lemma \ref{lem:sepedges}, for all $q'$ in $S_{R_\varepsilon}^{R+\varepsilon} \cap P$, if $d(q',q_{\overline{wv}}) \leq 3\varepsilon$, then $\|q'- q\|> 3 \varepsilon$. Hence, each edge contributes at least 1 connected component to  $\mathfrak{G}_{3\varepsilon}(P \cap S_{R-\varepsilon}^{R+\varepsilon}(p))$. 
		
		If there are more than 2, then $p$ has the $(R,\varepsilon)$-local structure of a vertex. We now assume there are 2 connected components $c_1, c_2$ (one per edge) in $\mathfrak{G}_{3\varepsilon}(P \cap S_{R-\varepsilon}^{R+\varepsilon}(p))$. Lemma \ref{lem:anglebound} gives
		
		\begin{align*}	
			\langle q_1 -p, q_2 -p \rangle &> -R^2+2R\varepsilon+7\varepsilon^2,\\
		\end{align*}
		
		and $p$ has the $(R,\varepsilon)$-local structure of a vertex.
    \end{proof}

    \begin{proposition}\label{prop:0balldeg2small}
		Let $v$ be a vertex of $|G| \subset \mathbb{R}^n$ with $\text{deg}(v) = 2$, with edges $\overline{uv}, \overline{wv}$. Let $P$ be an $\varepsilon$-sample of $|G|$. If $\angle uvw \leq \frac{\pi}{2}$, then for all  $p \in P$ with $\|p-v\| \leq \frac{R-\varepsilon}{2}$, $p$ has the $(R,\varepsilon)$-local structure of a vertex.

	\end{proposition}
	
	\begin{proof}
	    As in the proof of Proposition \ref{prop:0ball}, $\mathfrak{G}_{2\varepsilon}(P \cap B_{R+\varepsilon}(p))$ is connected. 
		
        For both edges $\overline{uv}, \overline{wv}$ there is at least one sample in $S_{R-\varepsilon}^{R+\varepsilon}(P)$, say $q_{\overline{uv}}$ and $q_{\overline{wv}}$. By Lemma \ref{lem:sepedges}, for all $q'$ in $S_{R_\varepsilon}^{R+\varepsilon} \cap P$, if $\|q'-q_{\overline{wv}}\| \leq 3\varepsilon$, then $\|q'- q\|> 3 \varepsilon$. Hence, each edge contributes at least 1 connected component to  $\mathfrak{G}_{3\varepsilon}(P \cap S_{R-\varepsilon}^{R+\varepsilon}(p))$. 

        If there are more than 2, then $p$ has the $(R,\varepsilon)$-local structure of a vertex. We now assume there are 2 connected components $c_1, c_2$ (one per edge) in $\mathfrak{G}_{3\varepsilon}(P \cap S_{R-\varepsilon}^{R+\varepsilon}(p))$. 
	   
	   Let $x_1$ and $x_2$ be points in $c_1$ and $c_2$. Without loss of generality, we have $d(x_1, \overline{uv}), d(x_2, \overline{wv}) \leq \varepsilon$. 
	   
	   From Lemma \ref{lem:anglebound} we know that $\langle x_1-p, x_2-p\rangle <-R^2+2R\varepsilon+7\varepsilon^2$. Since this inequality holds for all pairs $x_1, x_2$ in the connected components $c_1$ and $c_2$ we know it also holds for the averages $q_1$ and $q_2$. Thus we conclude $p$ has the $(R,\varepsilon)$-local structure of a vertex.

		\begin{figure}[H]
		\centering
		\begin{tikzpicture}[x=0.75pt,y=0.75pt,yscale=-1.5,xscale=1.5]
			\draw (0,0) -- (120,0);
			\draw (0,0) -- (0,-100);
			\draw (120,0) node [font=\small][anchor=west] [inner sep =0.75pt] {$w$};
			\draw (0,-100) node [font=\small][anchor=south east] [inner sep =0.75pt] {$u$};
			\draw (0,0) node [font=\small][anchor=north east] [inner sep =0.75pt] {$v$};
			\draw (10,-10) -- (-10,-77.5);
			\draw (10,-10) -- (110,10);
			\draw (10,-10) node [font=\small][anchor=south west] [inner sep =0.75pt] {$p$};
			\draw (-10,-77.5) node [font=\small][anchor=east] [inner sep =0.75pt] {$q_2$};
			\draw (110,10) node [font=\small][anchor=north] [inner sep =0.75pt] {$q_1$};
			\draw (-10,-77.5) -- (0, -77.5);
			\draw (0,-77.5) node [font=\small][anchor=south west] [inner sep =0.75pt] {$\widetilde{q_2}$};
			\draw(110,10) -- (110,0);
			\draw (110,0) node [font=\small][anchor = south][inner sep =0.75] {$\widetilde{q_1}$};
			\draw (10,-10) -- (10,0);
			\draw (10,0) node [font=\small][anchor = north][inner sep =0.75pt] {$\widetilde{p}$};
			\draw (10,0) -- (110,0);
		    \draw (10,0) -- (0,-77.5);
		\end{tikzpicture}
		\caption{}\label{fig:deg2lesspi2}
		\end{figure}

	\end{proof}

	\begin{proposition}\label{prop:nostrayverts}
		Let $p \in P$ be a sample with $\|p-v\| >\frac{3R+ \varepsilon}{2}$ for all vertices $v \in |G|$. Then $p$ has the $(R,\varepsilon)$-local structure of an edge.
	\end{proposition}
	
	\begin{proof}
    
        We begin by showing that if there is a sample $q \in S_{R-\varepsilon}^{R+\varepsilon}(p)\cap P$ with $d(q, \overline{uv} ) > \varepsilon$, then $ \mathfrak{G}_{3\varepsilon} \left( B_{R+\varepsilon}(p) \cap P\right)$ is disconnected. To prove this suppose not. Then there exists $x,y\in B_{R+\varepsilon}(p) \cap P$ such that $d(x, \overline{uv})<\varepsilon$, $d(y, \overline{uv})>\varepsilon$ and yet $\|x-y\|<3\varepsilon$. 
        
        This splits into two cases:
        
        \begin{enumerate}
            \item[(i)] $d(y,\overline{wv}) \leq \varepsilon$ for some vertex $w \neq u$ (noting that this case covers an edge $\overline{wu}$ as well),
            \item[(ii)]$d(y, \overline{wz})\leq \varepsilon$ for vertices $w,z\neq u,v$.
        \end{enumerate}
        
        For case (i), first observe that $\|x-v\|,\|y-v\|>\frac{R-\varepsilon}{2}$. We then get a contradiction via Lemma \ref{lem:sepedges} (with $D=\frac{R-\varepsilon}{2}$) using Assumption \ref{as:embed} (4).
        
        For case (ii) recall that Assumption \ref{as:embed} (3) implies $d(\overline{uv}, \overline{wz})>5\varepsilon$. However $d(\overline{uv}, \overline{wz})<d(\overline{uv}, x)+\|x-y\|+d(y, \overline{wz})\leq 5\varepsilon$ which is a contradiction.
      
        We thus conclude that if there is some $q \in S_{R-\varepsilon}^{R+\varepsilon}(p)\cap P$ with $d(q, \overline{uv} ) > \varepsilon$ then $\mathfrak{G}_{3\varepsilon}(P \cap B_{R+\varepsilon}(p))$ is disconnected and $p$ has the $(R,\varepsilon)$-local structure of an edge.
      
        We can now assume that $\mathfrak{G}_{3\varepsilon}(P \cap B_{R+\varepsilon}(p))$ is connected, and for all $q \in P \cap B_{R+\varepsilon}(p)$, $d(q, \overline{uv}) \leq \varepsilon$.
       
        We need to show that there are two clusters of samples in $S_{R-\varepsilon}^{R+\varepsilon}(p)$.
        Let $n \in \overline{uv}$ satisfy $\|p-n\| = R$, and assume that $n$ and $q$ are on the same side of the hyper-plane $H$ through $p$ perpendicular to $\overline{uv}$. Now let $\widetilde{p}, \widetilde{q}$ be the projections of $p$ and $q$ respectively to $\overline{uv}$. 
    
        We will split the analysis into the cases where $\|\tilde{p}-\tilde{q}\|\leq \|\tilde{p}- n\|$ and where $\|\tilde{p}- \tilde{q}\|> \|\tilde{p}- n\|$. 
        
        \begin{figure}[H]            
            \centering
		    \tikzset{every picture/.style={line width=0.75pt}} 

	        \begin{tikzpicture}[x=0.75pt,y=0.75pt,yscale=-1,xscale=1]
                \draw (-150,0) -- (250,0);
                \draw (0,-20) -- (200,0);
                \draw (0,-20) -- (120, 20);
                \draw (120,20) -- (200,0);
                \draw (0,0) -- (0,-20);
                \draw (120,20) -- (120,0);
                \draw (-150,0) node [anchor = east][font=\small][inner sep =0.75pt] {$v$};
                \draw (0,-20) node [anchor = south east][font=\small][inner sep =0.75pt] {$p$};
                \draw (0,0) node [anchor = north][font=\small][inner sep =0.75pt] {$\widetilde{p}$};
                \draw (200,0) node [anchor =  north][font=\small][inner sep =0.75pt] {$n$};
                \draw (250,0) node [anchor=west][font=\small][inner sep =0.75pt] {$u$};
                \draw (120,20) node [anchor = north][font=\small][inner sep =0.75pt] {$q$};
                \draw (120,0) node [anchor = north east][font=\small][inner sep =0.75pt] {$\widetilde{q}$};
            \end{tikzpicture}
            \caption{The case where $\|\widetilde{p}-\widetilde{q}\|<\|\widetilde{p}-n\|$.}
            \label{fig:edge2small}
        \end{figure}
        
       Consider $\|\widetilde{p}-\widetilde{q}\| \leq \|\widetilde{p}-n\|$, as in Figure \ref{fig:edge2small}. Note that 
       $\|\tilde{p}-n\|\leq R$ and 
    $\|\tilde{p}-\tilde{q}\|\geq \sqrt{(R-\varepsilon)^2-(2\varepsilon)^2}$ which implies 
        
          \begin{align}\label{eq:pqn}
            \|q-n\|^2 & = \|q-\widetilde{q}\|^2 + \left(\|\tilde{p}-n\| - \|\widetilde{p}- \widetilde{q}\|^2\right)\nonumber \\ 
            &\leq \varepsilon^2 + \left( R - \sqrt{(R-\varepsilon)^2 - 4\varepsilon^2}\right)^2 
        \end{align}
        
        Now consider $\|\widetilde{p}-n\| < \|\widetilde{p}-\widetilde{q}\|$, such as in Figure \ref{fig:edge2big}.
        Here we use the bounds $\|\tilde{p}-n\|\geq \sqrt{R^2-\varepsilon^2}$ and $\|\tilde{p}-\tilde{q}\|\leq R+\varepsilon$ to say
        
        \begin{align}\label{eq:pnq}
            \|q-n\|^2 &= \|q-\widetilde{q}\|^2 + \left( \|\widetilde{p}-\widetilde{q}\| - \|\widetilde{p}-n\| \right)^2 \nonumber \\
                        &\leq \varepsilon^2 + \left( \sqrt{(R + \varepsilon)^2 } -  \sqrt{R^2 -\varepsilon^2} \right)^2. 
        \end{align}
        
        Algebraic manipulation shows that both \eqref{eq:pqn} and \eqref{eq:pnq} are bounded from above by $4\varepsilon^2$ whenever $R>12\varepsilon$.
        
        \begin{figure}[H]            
            \centering
		    \tikzset{every picture/.style={line width=0.75pt}} 

	        \begin{tikzpicture}[x=0.75pt,y=0.75pt,yscale=-1,xscale=1]
                \draw (-150,0) -- (250,0);
                \draw (0,-20) -- (200,0);
                \draw (0,-20) -- (220, 20);
                \draw (220,20) -- (200,0);
                \draw (0,0) -- (0,-20);
                \draw (220,20) -- (220,0);
                \draw (-150,0) node [anchor = east][font=\small][inner sep =0.75pt] {$v$};
                \draw (0,-20) node [anchor = south east][font=\small][inner sep =0.75pt] {$p$};
                \draw (0,0) node [anchor = north][font=\small][inner sep =0.75pt] {$\widetilde{p}$};
                \draw (200,0) node [anchor =  north][font=\small][inner sep =0.75pt] {$n$};
                \draw (250,0) node [anchor=west][font=\small][inner sep =0.75pt] {$u$};
                \draw (220,20) node [anchor = north][font=\small][inner sep =0.75pt] {$q$};
                \draw (220,0) node [anchor = north west][font=\small][inner sep =0.75pt] {$\widetilde{q}$};
            \end{tikzpicture}
            \caption{The case where $\|\widetilde{p}-\widetilde{q}\|>\|\widetilde{p}-n\|$.}
            \label{fig:edge2big}
        \end{figure}
        
        Thus, for all $q$ on the same side of $H$ as $n$ with $\|p-q\| \leq R$, we have $\|q-n\| \leq 2 \varepsilon$. 
        
        As $n \in \overline{uv}$, there is a sample $q_n\in P$ with $\|n-q_n\| \leq \varepsilon$. Importantly since $B_{\varepsilon}(n)\subset S_{R-\varepsilon}^{R+\varepsilon}(p)$ we can say that $q_n$ connects to all the $P\cap S_{R-\varepsilon}^{R+\varepsilon}(p)$ on the same side of $H$  within $\mathfrak{G}_{3\varepsilon}\left(P \cap S_{R-\varepsilon}^{R+\varepsilon}(p)\right)$.
     
        Thus, on each side of $H$, we have a single cluster of points, which are connected at $3\varepsilon$. Thus, $\mathfrak{G}_{3\varepsilon}\left(P \cap S_{R-\varepsilon}^{R+\varepsilon}(p)\right)$ has two connected components. Then, Lemma \ref{lem:getedges} implies that $p$ has the $(R,\varepsilon)$-local structure of an edge.
    \end{proof}

\section{Algorithm and Its Correctness}\label{sec:alg}

	In this section, we present the algorithm from \href{http://github.com/yossibokor/Skyler.jl}{Skyler}, and prove that given $P$ an $\varepsilon$-sample of an embedded graph $|G|= (G, \phi_G)$ satisfying Assumptions \ref{as:embed}, the algorithm returns an isomorphic graph structure. The algorithm partitions $P$ into $P_0$ and $P_1$, such that for each $p \in P_0$, $p$ has the $(R, \varepsilon)$-local structure of a vertex, and for each $p \in P_1$, $p$ has the $(R, \varepsilon)$-local structure of an edge. We then detect the number of vertices, the number of edges and the boundary operator. To obtain $P_0$ and $P_1$, we use the function $\Delta_{R, \varepsilon}: P \to \{0,1\}$, (Algorithm \ref{alg:dim}),  such that if $p$ has $(R, \varepsilon)$-local structure of a vertex $\Delta_{R, \varepsilon}(p)=0$ and if $p$ $(R, \varepsilon)$-local structure of an edge, $\Delta_{R,\varepsilon}(p)=1$. Then,  $P_0 = \Delta_{R,\varepsilon}^{-1}(0)$ and $P_1 = \Delta_{R,\varepsilon}^{-1}(1)$.

	For each vertex $v \in |G|$, if $\text{deg}(v) \neq 2$, Proposition \ref{prop:0ball} implies that for all $p \in P$ with $\|p- v\| \leq \frac{R}{2}$, $\Delta_{R,\varepsilon}(p)=0$, while if $\text{deg}(v) = 2$, Propositions \ref{prop:0balldeg2big} and \ref{prop:0balldeg2small} imply that $\Delta_{R, \varepsilon}(p) =0$, and Proposition \ref{prop:nostrayverts} implies that if $\|p-v\| > \frac{3R}{2} + 2 \varepsilon$, $\Delta_{R, \varepsilon}(p)=1$.

    \begin{lemma}\label{lem:Q0}
        Let $x\in P_0$ and $\|x-v\|<\frac{3R}{2}+\varepsilon$ for vertex $v$. Then $y\in P_0$ is in the same connected component as $x$ in $\mathfrak{G}_{\frac{3R}{2} + 2\varepsilon}(P_0)$ if and only if $\|y-v\|<\frac{3R}{2}+\varepsilon$.
    \end{lemma}

    \begin{proof}
        By Proposition \ref{prop:nostrayverts} $P_0\subset P \cap\left\{ \bigcup_{v\in V} B_{ \frac{3R}{2} +\varepsilon}(v)\right\}$. Our embedding assumptions require that for vertices $v\neq v'$ we have $\|v-v'\|>\frac{9R}{2} +3\varepsilon$ and hence no points in $P\cap B_{ \frac{3R}{2} +\varepsilon}(v')$ are within $\frac{3R}{2}+\varepsilon$ of those in $B_{\frac{3R}{2} +\varepsilon}(v')$. This means they can not be connected in $\mathfrak{G}_{\frac{3R}{2} + 2\varepsilon}(P_0)$. This implies that the entire connected component containing $x$ must lie in $B_{\frac{3R}{2} +\varepsilon}(v).$  If $\|y-v\|>\frac{3R}{2}+\varepsilon$ then it cannot be in the same connected component as $x$.
    
        We finally wish to show that $\|y-v\|<\frac{3R}{2}+\varepsilon$ implies that $x$ and $y$ are in the same connected component.  Choose vertices $u_y$ and $u_x$ such that $d(y, \overline{u_yv})<\varepsilon$ and $d(x,\overline{u_xv})<\varepsilon$ and let $z_y\in \overline{uv}$ be the point $3\varepsilon$ from $v$ and analogously define $z_x$. As $P$ is an $\varepsilon$-sample of $|G|$ we have samples $p_y$ and $p_x$ such that $\|p_y- z_y\|<\varepsilon$ and $\|p_x- z_x\|<\varepsilon$. Note that $p_y, p_x\in P\cap B_{4\varepsilon}(v)$ and hence by Propositions \ref{prop:0ball} and \ref{prop:0balldeg2big} we know that $p_y, p_x \in P_0$. By construction $\|y - p_y\|, \|p_y - p_x\|$ and $\|p_x - x\|$ are all less that $\frac{3R}{2}+\varepsilon$ and hence $y$ and $x$ are in the same connected  component in  $\mathfrak{G}_{\frac{3R}{2} + 2\varepsilon}(P_0)$.
    \end{proof}

	The above lemma shows the correspondence between vertices in $G$ and connected components in 
	$\mathfrak{G}_{\frac{3R}{2} + 2\varepsilon}(P_0)$. Unfortunately the situation is less clean for the connected components of $\mathfrak{G}_{3\varepsilon}(P_1)$.	Around each vertex $v$ there is a `grey area', in which samples $p$ can be placed in either $P_0$ or $P_1$. Due to the size of this spherical shell, it is possible to obtain connected components in  $\mathfrak{G}_{3\varepsilon}(P_1)$ which contain points only within such a grey area. We devote the next few results to characterising the connected components of $\mathfrak{G}_{3\varepsilon}(P_1)$. We first show that every connected component of $\mathfrak{G}_{3\varepsilon}(P_1)$ is close to only one edge.

	\begin{proposition}\label{prop:P1uv}
        Let $[x]$ be a connected component of $\mathfrak{G}_{3\varepsilon}(P_1)$. Then there exists an edge $\overline{uv}$ such that $d(y, \overline{uv})<\varepsilon$ for all $y\in [x]$.

    \end{proposition}
    
    \begin{proof}
        Since every sample in $P$ is within $\varepsilon$ of some edge it is sufficient to show that if $p,q \in P_1$ with $d(p, \overline{uv}) \leq \varepsilon$ and $\|p-q\|\leq 3\varepsilon$ then $d(q,\overline{uv})<\varepsilon$.

        As $p \in P_1$, Propositions \ref{prop:0ball}, \ref{prop:0balldeg2big} and \ref{prop:0balldeg2small} imply
		
		\begin{enumerate}
		    \item for all vertices $w \in |G|$ with $\text{deg}(w) \neq 2$, $\|p-w\| >\frac{R-\varepsilon}{2}$,
		    \item for all vertices $w$ with $\text{deg}(w)=2$, $\|p - w\| \geq 4\varepsilon$.
		\end{enumerate}

		Without loss of generality, assume $\|p-v\| \leq \|p-u\|$. By Assumptions \ref{as:embed} (3) for all edges $\overline{xy}$ with $x,y$ distinct from $u,v$, $d(\overline{uv}, \overline{xy}) >5\varepsilon$. Hence, $d(p,\overline{xy}) > 4\varepsilon$, and for any sample $q$ with $d(q, \overline{xy}) \leq \varepsilon$, $\|p-q\| > 3 \varepsilon$. If $\text{deg}(v) \neq 2$, then $\|p-v\| > \frac{R-\varepsilon}{2}$, and as $|G|$ satisfies Assumptions \ref{as:embed} (4), Lemma \ref{lem:sepedges} implies $\|p-q\| > 3 \varepsilon$ for all $q \in P_1$ with $d(q,\overline{uv}) > \varepsilon$.
		
		Now assume that $\text{deg}(v) =2$, and  consider another edge $\overline{wv}$. For $\Phi(R,\varepsilon) \leq \angle uvw < \frac{\pi}{2}$ we can apply Lemma \ref{lem:sepedges} with $D=\frac{R-\varepsilon}{2}$  to see that for all $q \in P_1$ with $d(q,\overline{wv})\leq \varepsilon$, $\|q,-p\| > 3\varepsilon$.
		For $\frac{\pi}{2} \leq \angle uvw \leq \Psi(R,\varepsilon)$, we apply Lemma \ref{lem:sepedges} with $D=4\varepsilon$ and observe that $\pi/2>\arccos(23/32) +2\arcsin(1/4)$ to conclude that $d(q,\overline{wv})\leq \varepsilon$, $\|q-p\| > 3\varepsilon$. 
    \end{proof}

    There can be multiple connected component in $\mathfrak{G}_{3\varepsilon}(P_1)$ near the same edge. However there will only be one which contains a sample near the midpoint of the edge. We wish to treat these differently and so we will give them a name.

    \begin{definition}
        We say that the connected component of $\mathfrak{G}_{3\varepsilon}(P_1)$ \emph{spans} the edge $\overline{uv}$ if it contains a point within $\varepsilon$ of the midpoint of $\overline{uv}$. Without reference to the specific edge $\overline{uv}$ we say that the component is \emph{spanning}.
    \end{definition}

    \begin{proposition}\label{prop:P1}
        Let $\overline{uv}$ be an edge in $G$. There exists a unique connected component $A_{\overline{uv}}$ which spans $\overline{uv}$. $A_{\overline{uv}}$ contains samples in both $B_{ \frac{3R+5\varepsilon}{2}}(u)$ and $B_{\frac{3R+5\varepsilon}{2}}(v)$.

        If $[x]\neq A_{\overline{uv}}$ is a connected component in $\mathfrak{G}_{3\varepsilon}(P_1)$ within $\varepsilon$ of $\overline{uv}$ then either $[x]\subset B_{\frac{3R+\varepsilon}{2}}(u)$ or $[x]\subset B_{\frac{3R+\varepsilon}{2}}(v)$.
    \end{proposition}
    
    \begin{proof}
        Let $m$ denote the midpoint of $\overline{uv}$.
    
        Let $t_0, t_1, \ldots t_{2M}$ be consecutive points along $\overline{uv}$ with $\|t_i- t_{i+1}\|<\varepsilon$, $\|t_0- u\|=\frac{3R+3\varepsilon}{2}$, $t_M=m$,  and $\|t_{2M}- v\|=\frac{3R+3\varepsilon}{2}$. There must be $z_0, z_1 z_2, \ldots z_M\in P$ such that $\|t_i- z_i\|<\varepsilon$. Observe that $\|z_i- u\|>\frac{3R+\varepsilon}{2}$ and $\|z_i- v\|>\frac{3R+\varepsilon}{2}$ and so by Proposition \ref{prop:nostrayverts} all the $z_i$ are in $P_1$. Since $\|z_i- z_{i+1}\|<3\varepsilon$ we know that all the $z_i$ lie in the same connected component of $\mathfrak{G}_{3\varepsilon}(P_1)$ which spans $\overline{uv}$ as $z_M$ is within $\varepsilon$ of $m$. 
    
        To see this connected component is unique we need only observe that any pair of samples in $P_1$ both within $\varepsilon$ of $m$ are within $3\varepsilon$ of each other and hence lie in the same connected component.
    Denote this unique connected component by $A_{\overline{uv}}$.
    
        Observe that $\|u- z_0\|<\frac{3R+3\varepsilon}{2}+\varepsilon$ and $\|v- z_{2M}\|<\frac{3R+3\varepsilon}{2}+\varepsilon$.

        Now suppose that $[x]\neq  A_{\overline{uv}}$ is a connected component in $\mathfrak{G}_{3\varepsilon}(P_1)$ within $\varepsilon$ of $\overline{uv}$. Since $[x]\neq A_{\overline{uv}}$, we have $d([x], t_i)>2 \varepsilon$ for all $i$ and hence 
        
        \begin{equation*}
            [x]\subset B_{\frac{3R+\varepsilon}{2}}(u)\cup B_{\frac{3R+\varepsilon}{2}}(v).
        \end{equation*}
        
        As $\|u-v\|>\frac{3R+\varepsilon}{2} + \frac{3R+\varepsilon}{2}+3\varepsilon$ we further conclude that $[x]$ is contained in only one of $B_{\frac{3R+\varepsilon}{2}}(u)$ or $B_{\frac{3R+\varepsilon}{2}}(v)$.
    \end{proof}

    In light of Proposition \ref{prop:P1} we modify our partition of $P$, into $\widetilde{P_0}$ and $\widetilde{P_1}$, see Definition \ref{def:tildes} and Algorithm \ref{alg:abstract}. We effectively want to move any points in $P_1$ that are not contained in a spanning connected component into $P_0$. 
	
	\begin{definition}[$\widetilde{P_0}$ and $\widetilde{P_1}$]\label{def:tildes}
	    Let $P$ be an $\varepsilon$-sample of an embedded graph $|G|$ satisfying Assumptions \ref{as:embed}, and consider the sets $P_0$ and $P_1$ from Definition \ref{defn:p0p1}. Let $Q_0$ be the connected components of $\mathfrak{G}_{\frac{3R}{2} + 2\varepsilon}(P_0)$, and $Q_1$ the connected components of $\mathfrak{G}_{3\varepsilon}(P_1)$, and define
	        $f: Q_1 \to \{0,1\}$ by
	    $f([q])=0$ when there is only a single connected component $[p]\in Q_0$ such that $d([p],[q])<3\varepsilon$, and $f([q])=1$ otherwise. 
	    
	    We define $\widetilde{P_0}\coloneqq P_0 \cup \left(\bigcup_{f([x])=0} [x]\right)$ and $\widetilde{P_1}\coloneqq  \left(\bigcup_{f([x])=1} [x]\right)$.
    \end{definition}
   
    \begin{lemma}\label{lem:tilde}
        Let $[x]\in Q_1$. Then $f([x])=1$ if and only is $[x]$ spans an edge, and $f([x])=0$ if and only if $[x]\subset B_{\frac{3R+\varepsilon}{2}}(v)$ for some vertex $v$. 
    \end{lemma}
    
    \begin{proof}
        If $[x]$ spans an edge $\overline{uv}$ then by Proposition \ref{prop:P1uv} we know that $[x]$ contains samples in both $B_{\frac{3R+5\varepsilon}{2}}(u)$ and $B_{ \frac{3R+5\varepsilon}{2}}(v)$. Let $x_u\in [x]$ be the sample closest to $u$. Note that $\|x_u-u\|\leq \frac{3R+5\varepsilon}{2}$. There must be some sample $p_u\in P$ with $\|p-u\|\in < \|u- x_u\|$ and $\|p-x_u\|<3\varepsilon$. Now $p\in P_0$ as otherwise it contradicts $x_u$ being the closest sample to $u$ inside $[x]$. By Lemma \ref{lem:Q0}, $[p_u]\in Q_0$ is contained in $B_{\frac{3R +\varepsilon}{2}}(u)$. 

        Similarly we can show that there some $x_v\in [x]$ and $p_v\in P_0$ with $\|x_v- p_v\|\leq 3\varepsilon$ and $[p_v]\in Q_0$ contained in $B_{\frac{3R+\varepsilon}{2}}(v).$ By Lemma \ref{lem:Q0} $[p_u]$ and $[p_v]$ are distinct and hence $f([x])=1$.

        If $[x]$ does not span any edge then by Proposition \ref{prop:P1uv} we know there is a vertex $v$ such that $[x]\subset B_{\frac{3R+\varepsilon}{2}}(v)$. We then can appeal to Lemma \ref{lem:Q0} to say that there is only one connected component in $Q_0$ within $3\varepsilon$ of $[x]$. 
    \end{proof}

    Let $\widetilde{Q_0}$ denote the connected components of $\mathfrak{G}_{\frac{3R}{2} + 2\varepsilon}(\widetilde{P_0})$ and let $\widetilde{Q_1}$ denote the connected components of $\mathfrak{G}_{3\varepsilon}(\widetilde{P_1})$. We will see that characterisation of the elements of $\widetilde{Q_0}$ is the same as that of $Q_0$. The elements of $\widetilde{Q_1}$ are exactly those connected components in that span some edge.
 
    \begin{theorem}\label{thm:bij}
        For each vertex $v$ there exists a unique connected component $[x]\in\mathfrak{G}_{\frac{3R}{2} + 2\varepsilon}(\widetilde{P_0})$  such that $[x]\subset B_{\frac{3R}{2} + 2\varepsilon}(v)$. Every connected component of $\mathfrak{G}_{\frac{3R}{2} + 2\varepsilon}(\widetilde{P_0})$ is of this form.

        For each edge $\overline{uv}$ there exists a unique connected component $[x]\in\mathfrak{G}_{3\varepsilon}(\widetilde{P_1})$  such that $[x]$ spans $\overline{uv}$. Furthermore every connected component of $\mathfrak{G}_{\frac{3R}{2} + 2\varepsilon}(\widetilde{P_1})$ is of this form.
    \end{theorem}

    \begin{proof}
        From Proposition \ref{prop:nostrayverts} and Lemma \ref{lem:tilde} we know that $\widetilde{P_0}\subset \bigcup_{v} B_{\frac{3R+\varepsilon}{2}}(v).$ We can then effectively repeat the proof of Lemma \ref{lem:Q0} to show the analogous result for $\widetilde{P_0}$. 
    
        To see the bijection between the vertices of $G$ and $\widetilde{Q_0}$ observe that every sample within $4\varepsilon$ of some vertex is in $P_0\subset \widetilde{P_0}$ and hence every vertex corresponds to some connected component, and observe that by Lemma \ref{lem:tilde} all points in $\widetilde{P_0}$ lie within $\frac{3R+\varepsilon}{2}$ of some vertex. 
    
        The characterisation for connected components of $\mathfrak{G}_{3\varepsilon}(\widetilde{P_1})$ follows directly from Proposition \ref{prop:P1uv} and Lemma \ref{lem:tilde}.
    \end{proof}

    Define the map $F_0: \widetilde{Q_0} \to V$ by $F_0([x])=\text{argmin}_{v\in V}\{d([x],v)\}$ and $F_1:\widetilde{Q_1}\to E$  by $F_1([x])=\text{argmin}_{\overline{uv}\in E}\{d([x], \text{midpt}(\overline{uv})\}$.

    That $F_0$ and $F_1$ are well defined bijections follows directly from Theorem \ref{thm:bij}. From Proposition \ref{prop:P1uv} we further can say that if $[q]\in \widetilde{Q_1}$ and $[x]\in \widetilde{Q_0}$ then the single linkage distance between $[q]$ and $[x]$ is less than $3\varepsilon$ if and only if $F_0([x])\in \partial_G(F_1([q])$.

	\begin{algorithm}\caption{$\Delta_{R, \varepsilon}(p)$}\label{alg:dim}
    		\KwData{An $\varepsilon$-dense sample $P$ of an embedded graph $|G|$, a point $p \in P$.} \KwResult{0 if $p$ has local structure of a vertex, 1 if $p$ has local structure of an edge. }
    		\Begin{
      			$\mathcal{G}_p\longleftarrow \{ q \in P \mid \|p-q\| \le R+\varepsilon\}$\;
				connect $q, q' \in \mathcal{G}_p$ if $\|q-q'\| \le 3 \varepsilon$\;
      			\If {$\mathcal{G}_p$ is disconnected}{\Return {1}}
      			\Else {remove $q \in \mathcal{G}_p$ if $\|p-q\| \le R -\varepsilon$\;
            	    \If{number of connected components in $\mathcal{G}_p$ is not $2$}{\Return{0}}
            		\Else{find the midpoints $q_1, q_2$ of the connected components $c_1$ and $c_2$\;
    			    \If{$\langle q_1 -p, q_2 -p\rangle > -R^2 +2 R\varepsilon - 7\varepsilon^2$}{\Return {0}}
              		    \Else{\Return {1}}
            		}
    			}	
			}
 		\end{algorithm}
	
	\begin{algorithm}\caption{Abstract Structure}\label{alg:abstract}
    		\KwData{Partition of $P$ into $P_0$ and $P_1$.} \KwResult{Partitions $\widetilde{P_0}, \widetilde{P_1}$, abstract graph $G=(E,V)$.}
    		\Begin{
			$E \longleftarrow \emptyset$\;
			$V \longleftarrow \emptyset$\;
			$\widetilde{P_0} \longleftarrow P_0$\;
			$\widetilde{P_1} \longleftarrow P_1$\;
      			\For{connected components $[p] \in \mathfrak{G}_{\frac{3R}{2}+2\varepsilon}(P_0)$}{add $[p]$ to $V$}
			\For{connected components $[q] \in \mathfrak{G}_{3\varepsilon}(P_1)$}{$B_q\longleftarrow \emptyset$\;
				\For{$[p] \in V$}{
					\If{$\min_{p' \in [p], q' \in [q]}\|p'-q'\| \leq 3\varepsilon$}{add $[p]$ to $B_{q}$}}}
				\If{size$(B_q)=1$}{add all $q'  \in [q]$ to $\widetilde{P_0}$ and remove them from $\widetilde{P_1}$}
				\Else{add $B_q$ to $E$}
                }	
     		\Return{$\widetilde{P_0}, \widetilde{P_1}, V,E$}
  	\end{algorithm}

  	\begin{algorithm}\caption{Expectation Maximisation for Vertex Location Prediction}\label{alg:vertex}
		\KwData{$|P|$ data points in $n$ dimensions, $N_0 + N_1 = N$ many strata pieces.} 
		\KwResult{Predicted embedded graph vertex locations.}
		\KwIn{Abstract graph structure.}
		\Begin{
			Initialise vertex locations $V$ \;
			Initialise $|P| \times N$ strata assignment matrix $A$ \;
			\For{$s_i$ in strata pieces $S = V \cup E$, $x_j$ in data points}{
				\If{$x_j \in s_i$}{
					$A_{i,j} \longleftarrow 1$}
				\Else{$A_{i,j} \longleftarrow 0$ }
			
			}
			assign an error threshold $\sigma \in \mathbb{R}_+$\;
			Initialise $\pi_i = \frac{\sum_i A_{i,j}}{\sum_{i,j} A_{i,j}}$\;
		
			\For{iterations in EM-iterations}{
			\For{$s_i$ in strata pieces $S = V \cup E$, $x_j$ in data points}{
				assign $A_{i,j} =  \mathbb{E}(1_{Z_j = 1} | X_j = x_j )$ through \eqref{Eq:A} \;
			}
			assign $\pi_i = \frac{\sum_i A_{i,j}}{\sum_{i,j} A_{i,j}}$\;
			assign $V = \arg \min_{V} V \rightarrow C(V,\Pi ; \sigma)$ \eqref{Eq:cost} through a hill climbing optimiser such as gradient-descent\;
			}
		}
	\end{algorithm}
	
\section{Vertex Prediction}\label{sec:modelling}
    Thus far, the focus has been on finding the abstract structure of an embedded graph $|G|$. We now aim to form a numerical scheme to estimate the vertex locations of $|G| \subset \mathbb{R}^n$. 
    In \cite{modsim2019}, a non-linear least-squares method was proposed and used for embedded graph reconstruction. Empirical observation of this method showed vertex predictions were often not contained in an $\varepsilon$-data sample of the \textit{true} embedded graph. 
    A point of difficulty here was that data that should belong to a one-dimensional strata piece was often assigned to a zero-dimensional strata when nearby a vertex location. 
    We utilise an Expectation-Maximisation (EM) algorithm which updates both the predicted vertex locations, and their strata assignments to correct this issue. 
    To do this we design a likelihood function with latent variables for strata assignment so that we may reconstruct a probability measure over the embedded graph from which our data is sampled. 
    Ideally, we would reconstruct a measure $\nu$ whose support is the embedded graph. 
    Recorded data has errors and makes it computationally infeasible to reconstruct $\nu$ directly. 
    Instead, we will formulate an approximating measure $\nu_{\delta}$ which satisfies:

    \begin{enumerate}
    	\item $\nu_{\delta}$ is equivalent to Lebesgue measure,
	    \item $\text{supp} (\lim_{\delta \rightarrow 0} \nu_{\delta} ) = |G|$, 
    \end{enumerate} 
    
    where the limit is meant in the weak sense. The first assumption gives robustness to measurement errors, and the second ensures that in ideal circumstances we form a measure that is supported on the embedded graph. 
    There are many measures which obey these conditions, we choose a Gaussian convolution model for each strata piece and combine all the strata pieces together through a categorical mixture model.

\subsection{Embedded Graph Model}

    Let $(\Omega, \mathcal{F} , \mu)$ be a probability space, that is $\Omega$ is a set, $\mathcal{F}$ is a $\sigma$-algebra of sets from $\Omega$, and $\mu : \mathcal{F} \mapsto [0,1]$ is a normalised measure.

    \begin{definition}
        Given a probability space $(\Omega, \mathcal{F} , \mu)$ and a field with a $\sigma$-algebra $\mathcal{B}$, a measurable function $f: (\Omega, \mathcal{F} , \mu) \mapsto (\mathbb{F}, \mathcal{B})$ is a \emph{random variable}. A vector valued random element is a vector valued measurable function $\tilde{f} : (\Omega, \mathcal{F} , \mu) \mapsto  (\mathbb{R}^n, \mathcal{B}(\mathbb{R}^n) )$ given through $\tilde{f} = (f_1, \dots, f_n)$ where each of the $f_i$ are random variables. 
    
        The \emph{expectation} of a random variable is the integral, $\mathbb{E}(f) \coloneqq \int_{\Omega} f d\mu$. Given a sub-$\sigma$-algebra $C \subset \mathcal{F}$, the \emph{conditional expectation} of a random variable $f$, $\mathbb{E}(f|C) \in L^2(\Omega, \mathcal{F} , \mu)$, is the unique function that satisfies
        
        \begin{align*}
            \int_B \mathbb{E}(f|C) d\mu = \int_B f d\mu,
        \end{align*}
        
        for all $B \in C$. The expectation and conditional expectation of a vector valued random element $\tilde{f} = (f_1, \dots, f_n)$ is defined component-wise through each of the random variables $f_i$, that is 
        
        \begin{align}
            \mathbb{E}(\tilde{f} |C) \coloneqq ( \mathbb{E}( f_1 |C), \dots, \mathbb{E}(f_n |C) )
        \end{align}
        
        for all $C \in \mathcal{B}(\mathbb{R}^n)$. 
    \end{definition}

    Above, we have adopted the standard notation $\mathcal{B}(\mathbb{R}^n)$ for the Borel-$\sigma$-algebra generated by the open sets in the standard topology on $\mathbb{R}^n$. Let $X_j: (\Omega, \mathcal{F} , \mu) \mapsto (\mathbb{R}^n, \mathcal{B}(\mathbb{R}^n)  )$ be vector valued random elements and $Z_j : (\Omega, \mathcal{F} , \mu) \mapsto ([N], 2^{[N]})$ be random variables for $j \in \{1,\dots,{\color{red} |P|}\}$, where $[N] \coloneqq \{1,\dots, N \}$, $n$ is the dimension of the space to which the graph is embedded, $|P|$ is the amount of recorded data points, and $N$ the number of strata in $|G|$. Let $N_0, N_1 \in \mathbb{N}_0$ be $N_0$ and $N_1$ are the number of zero and one dimensional strata respectively, and so $N = N_0 + N_1$.
    
    Enumerate the set of vertex locations as $V \coloneqq \{v_i\}_{i=1}^{N_0}$. For each $i \in \{N_0+1, \dots, N\}$ assign the pairing $v_{i_1},v_{i_2} \in \{v_i\}_{i=1}^{N_0}$ to be the vertices that form the boundary $i^{th}$ strata piece. Assume that for each $j$ the $Z_j$ are independent and identically distributed, that each $X_j$ is independent of $X_i$ and $Z_i$ for $i\neq j$. 

    We place the following constraints on the random variables:
    
    \begin{enumerate}
	    \item $Z_j \sim \text{Categorical}(\Pi)$ with parameters $\Pi \coloneqq (\pi_1, \dots, \pi_N)$,
	    \item $\mathbb{E}( X_j | Z_j = i ) \sim \ \text{Normal} (v_j, \sigma_j)$ for $j\in \{ 1, \dots, N_0 \}$,
	    \item $\mathbb{E}( X_j | Z_j = i  ) = t_j v_{i_1} + (1-t_j) v_{i_2} + \varepsilon$ where $t_j \sim \text{Uniform}([0,1])$ and \\
	    $\varepsilon_j \sim \text{Normal} (0, \sigma_i)$ for $i \in \{ N_0 + 1, \dots, N \}$.
    \end{enumerate}
    
    The categorical random variables $Z_j$ represent which strata piece a random element $X_j$ belongs to. The categorical distribution is defined on $N$ many categories, with the $i^{th}$ category having a probability of $\pi_i$ of being observed. In our case, each $\pi_i$ represents approximately how many data points belong the $i^{th}$ strata piece.\\ 

    The distribution of $\mathbb{E}( X_j | Z_j = i \in \{ N_0 + 1, \dots, N \}  ) = t v_{j_1} + (1-t) v_{j_2} + \varepsilon$ is
    
    \begin{align}\label{Eq:den}
        \rho_{v_{i_1},v_{i_2}}(x ; \sigma_i) \coloneqq \frac{1}{({2 \pi \sigma_i^2})^{n/2}} \int_{0}^{1} e^{-\|x - (t v_{i_1} + (1-t) v_{i_2}) \|^2_2/2\sigma_i^2} dt,
    \end{align}
   
    where $\rho ( \ \cdot \ ; 0, \sigma_i)$ is a normal density in $n$ dimensions with zero mean and variance $\sigma_i^2$. This can be obtained through noting that if $\nu_{v_{i_1},v_{i_2}}$ is uniform measure on $ \mathcal{L}_{v_{i_1} ,v_{i_2} } \coloneqq \{y \ | \ y = t v_{i_1} + (1-t) v_{i_2}, \ t \in [0,1] \}$, then the  measure $\nu_{\sigma_i,v_{i_1},v_{i_2}} = \rho_{v_{i_1},v_{i_2}}(x ; \sigma_i) dx$ is given through
    \begin{align*}
        \nu_{\sigma_i,v_{i_1},v_{i_2}} &=  \rho ( x ; 0, \sigma_i)dx * \nu_{v_{i_1},v_{i_2}}\\
        &=\frac{1}{({2 \pi \sigma_i^2 })^{n/2}} \int_{0}^{1} e^{-\|x - (t v_{i_1} + (1-t) v_{i_2}) \|^2_2/2\sigma_i^2} dt dx,
    \end{align*}
    where $*$ represents the convolution operation over measures. 
    Through this convolution construction, we have the following proposition. 

    \begin{proposition}
	    Let $\rho_{v_{i_1},v_{i_2}}$ and $\nu_{v_{i_1},v_{i_2}}$ be as given above, then:
	    \begin{enumerate}
		    \item $\rho_{v_{i_1},v_{i_2}} \in C^{\infty} (\mathbb{R}^n)$,
		    \item $\rho_{v_{i_1},v_{i_2}} dx$ is equivalent to Lebesgue measure,
		    \item and $\nu_{\sigma_i,v_{i_1},v_{i_2}} \xrightarrow{\sigma_i \rightarrow 0 } \nu_{v_{i_1},v_{i_2}}$ weakly. 
	    \end{enumerate}
    \end{proposition} 
    
    The first two claims follow from Equation \ref{Eq:den}. The third is a result from mollifier approximation theory, see \cite{stein2009real} for details. 

    \begin{corollary}\label{cor:allgood}
	    Define $\sigma \coloneqq \max_{i} \sigma_i$ and let $\nu_{\sigma} \coloneqq \mu( X_j^{-1} )$ be the push-forward measure of $\mu$ through $X_j$, then 
	    \begin{enumerate}
		    \item $\nu_{\sigma} \sim dx$,
		    \item $ \text{supp}(\lim_{\sigma \rightarrow 0} \nu_{\sigma}) = |G|$,
		    \item $\lim_{\sigma \rightarrow 0 }\nu_{\sigma} (|G|) = 1$,
	    \end{enumerate}
	    where $|G|$ is the embedded graph in $\mathbb{R}^n$.
    \end{corollary}
    
    \begin{proof}
    	Write $\nu_{\sigma}$ through
    	
	    \begin{align*}
	        \nu_{\sigma} &= \sum_{i=1}^{N_0} \pi_i \rho (x; v_i, \sigma_i) dx + \sum_{i=N_0 + 1}^{N} \pi_i \rho_{v_{j_1},v_{j_2}}(x ; \sigma_i) dx\\
        	&= \sum_{i=1}^{N_0} \pi_i ( \delta_{v_i} * \rho ( x ; 0, \sigma_i) dx) + \sum_{i=N_0 + 1}^{N} \pi_i (\nu_{v_{i_1},v_{i_2}} * \rho ( x ; 0, \sigma_i)dx )
	    \end{align*}
	    
	    where $\delta_{v_i}$ is the normalised measure: $\delta_{v_i}(U) = 1$ if $v_i \in U$ and zero otherwise. Let $|G|$ be the embedded graph and define $|G|_r \coloneqq \{ x \ | \ x \in B_r(y), \ y \in |G| \}$, then
	    
	    \begin{align*}
	        \nu_{\sigma} (\mathbb{R}^n \setminus |G|_r )	&\leq \int_{\mathbb{R}^n \setminus |G|_r }\left( \sum_{i=1}^{N_0} \pi_i \delta_{v_i} + \sum_{i=N_0 + 1}^{N} \pi_i \nu_{v_{i_1},v_{i_2}} \right) * \rho ( x ; 0, \sigma)dx\\
	        &\xrightarrow{\sigma \rightarrow 0 } 0 \quad \text{for all $r>0$.}
	    \end{align*}
    \end{proof}

    Corollary \ref{cor:allgood} shows the push-forward measure $\nu$ has our desired properties for modelling an embedded graph $|G|$.

\subsection{Parameter re-Estimation}

    We now form an Expectation Maximisation (EM) algorithm to find Maximum Likelihood Estimates (MLEs) for the embedded graph's vertex locations. Let $\widetilde{\mathbb{P}}(\Omega)$ be the space of probability measures over $\Omega$. We are interested in reconstructing the measure $\mu$ given evaluations of $X_j$ and $Z_j$ for every $j\in \{1,\dots,n\}$. This forms the following likelihood optimisation problem:
    
    \begin{equation*}
        \mu^* \coloneqq \text{argsup}_{\eta \in \widetilde{\mathbb{P}}(\Omega)} \eta \left( \bigcap_{j \in \{1, \dots, |P| \} } X_j^{-1} (B_h(x_j)) \cap Z^{-1}_j (i) \right)
    \end{equation*}
    
    for some small $h >0 $. For a single recorded datum:
    
    \begin{align*}
        &\eta ( X_j^{-1} (B_h(x_j)) \cap Z^{-1}_j (i) ) =\mathbb{P}( X_j \in B_h(x_j) \ | \ Z_j = i ) \mathbb{P}(Z_j = i)\\
        &= \prod_{i=1}^{N_0} \left( \pi_i \int_{B_h(x_j)}  \rho(x ; v_i, \sigma_i )dx \right)^{1_{Z_j= i }}  \prod_{i=N_0+1}^{N}\left( \pi_i \int_{B_h(x_j)} \rho_{v_{j_1},v_{j_2}}(x; \sigma_i)dx \right)^{1_{Z_j = i}}.
    \end{align*}

    Intersecting over all such data points, taking a logarithm, and evaluating the limit as $h \rightarrow 0$ for the argument supremum yields the equivalent optimisation:

    \begin{align}\label{Eq:opt}
         \text{argsup}_{\pi_i \in [0,1], \  v_i \in \mathbb{R}^n} \sum_{j=1}^{|P|}  \hspace{-0.2em} &\Big(  \sum_{i=1}^{N_0}  1_{Z_j=i} (\log( \rho (x_j; v_i, \sigma_i) ) + \log(\pi_i))+ \\
        &\sum_{i=N_0+1}^{N}  1_{Z_j = i} (\log( \rho_{v_{i_1},v_{i_2}}(x_j;\sigma_i) ) + \log(\pi_i)) \Big). \nonumber
    \end{align}

    We cannot observe accurately $Z_j$ for a recorded datum, although the work in estimating the abstract graph structure gives an initial estimate for this value. To dynamically update the prediction of this value, we will utilise an EM-algorithm. Projection to the sub-$\sigma$-algebra $\sigma (X_1, \dots, X_n)$ and making the assumption $Z_j \perp X_{\widetilde{j}}$ for $\widetilde{j} \neq j$ gives the following log-likelihood function, which we aim to maximise: 
    
    \begin{align}\label{Eq:cost}
        \mathcal{L}(V, \Pi; \sigma) &\coloneqq\\ \nonumber
        \frac{1}{|P|}\sum_{j=1}^{|P|}  \Big(  &\sum_{i=1}^{N_0} \mathbb{E}( 1_{Z_j = i} | X_j \in B_{h'}(x_j) ) (\log( \rho (x_j; v_i, \sigma_i) ) + \log(\pi_i))+  \\
        &\sum_{i=N_0+1}^{N} \mathbb{E}( 1_{Z_j = i} | X_j \in B_{h'}(x_j) ) (\log( \rho_{v_{i_1},v_{i_2}}(x_j;\sigma_i) ) + \log(\pi_i)) \Big), \nonumber
    \end{align}
    
    where we currently view $\mathcal{L}$ as a function of the vertex locations $V$ and assignment weights $\Pi$, with $\sigma$ being a fixed value. Let the densities for each $k \in \{1,\dots , N\}$ strata be enumerated as $\{ \rho_k \}_{k=1}^N$. The individual terms of the cost function are 

    \begin{align}\label{Eq:A}
        \lim_{h' \rightarrow 0}\mathbb{E}( 1_{Z_j = i}  | X_j \in B_{h'}(x_j))  
        &=  \frac{\pi_i \rho_i(x_j) }{\sum_{k=1}^{N} \pi_k \rho_k(x_j) }  \\
        \log( \rho(x; v_i , \sigma_i ) ) &= - \frac{d}{2}\log (2 \pi \sigma_i) -\|x - v_i \|^2/2\sigma_i. \nonumber
    \end{align}

    \vspace{-1em}
    \begin{align*}
        \log( \rho_{v_{i_1},v_{i_2}}(x;\sigma_i) ) &= 
        \log \Big(\text{erf}\left(\frac{ \langle v_{i_1} - v_{i_2}, v_{i_1} + v_{i_2} - 2x \rangle+\|v_{i_1}- v_{i_2}\|_2^2}{2 \sqrt{2} \|v_{i_1} - v_{i_2}\|_2 \sigma_i}\right)\\
        &\ -\text{erf}\left(\frac{ \langle v_{i_1} - v_{i_2}, v_{i_1} + v_{i_2} - 2x \rangle-\|v_{i_1} - v_{i_2}\|_2^2}{2 \sqrt{2} \|v_{i_1} - v_{i_2}\|_2 \sigma_i}\right)\Big) \\
        &\ + \frac{ \langle v_{i_1} - v_{i_2}, v_{i_1} + v_{i_2} - 2x \rangle^2-4 \|v_{i_1} - v_{i_2}\|_2^2 \|(v_{i_1} + v_{i_2})/2 - x\|_2^2}{8 \|v_{i_1} - v_{i_2}\|_2^2 \sigma_i^2}\\
        &\ -\log (\|v_{i_1} - v_{i_2}\|_2)+\log \left(2^{\frac{1}{2} (-d-1)} \pi ^{\frac{1}{2}-\frac{d}{2}} \sigma_i^{1-d}\right) .
    \end{align*}

    Above, $\text{erf} : \mathbb{R} \mapsto \mathbb{R}$ is the standard error function given through $\text{erf}(x) = \frac{2}{\sqrt{\pi}} \int_{0}^{x} \exp(-t^2) dt$. In \href{http://github.com/yossibokor/Skyler.jl}{Skyler}, the analytic gradients of the log-likelhood function $\mathcal{L}$ are given. Gradient clipping is used to bound our computations within machine accuracy for when $\sigma_i$ or the evaluation of $x \mapsto \rho_{v_{i_1},v_{i_2}}(x;\sigma_i)$ is close to machine precision. Our log-likelhood function is often not concave, for instance the function $(v_{i_1}, v_{i_2} ) \mapsto  \rho_{v_{i_1},v_{i_2}}(x;\sigma_i) $ obeys $\rho_{v_{i_1},v_{i_2}}(x;\sigma_i) = \rho_{v_{i_2},v_{i_1} }(x;\sigma_i) $. It is necessary to have a good initialisation for the embedded graph modelling to find an acceptable local optimum value for vertex prediction. In our computations, we have found that the initial vertex modelling given by the abstract graph structure yields vertex predictions with an error less than the noise of the data, correcting the issue observed in \cite{modsim2019}. We can complete Algorithm \ref{alg:vertex} by noting that if $A_{i,j} \coloneqq \lim_{h' \rightarrow 0}\mathbb{E}( 1_{Z_j = i}  | X_j \in B_{h'}(x_j))$, then the function $\Pi \rightarrow \mathcal{L}(V,\Pi; \sigma)$ is concave and has a unique maximum value at $\pi_i^* = \frac{\sum_{j} A_{i,j}}{\sum_{i,j}A_{i,j}  }$. It can be seen that our model is a higher-dimension version of Gaussian clustering as Algorithm \ref{alg:vertex} degenerates to this when $N = N_0$.\\

    Fixing a noise tolerance $\sigma$ and solving the optimisation in Equation \ref{Eq:opt} by minimising the function $(V, \Pi) \rightarrow \mathcal{L}(V,\sigma, \Pi)$ through an EM-algorithm \cite{dempster1977maximum} gives Algorithm \ref{alg:vertex}.

\subsection{Numerical Simulations}

    The conditions in Assumption \ref{as:embed} are not the sharpest bounds, and other ratios of $R$ and $\varepsilon$ can also detect the correct graph structure. We present the results of a few different ratios, for the same $0.1$-sample $P$ (Figure \ref{fig:sample}) of the embedded graph $(G, \phi_G) \subset \mathbb{R}^3$ (Figure \ref{fig:graph}).  There are $705$ samples in $P$, and $G$ has 5 vertices embedded as 1: $(0,0,0)$, 2: $(4.6, 6.24, 0)$ 3: $(4.86, 0.51, 3.47)$, 4: $(-1.32, 6.29, 4)$, and 5: $(-4.23, -3.48, -3)$, and edges
        
        \begin{equation*}
            E = \{ (1,5), (1,3), (1,4), (2,4), (2,3)\}.
        \end{equation*}
        
    The following table describes the results with varying choices of ratio $\frac{R}{\varepsilon}.$

    \begin{table}[]
        \begin{tabular}{cccccccc}
            Ratio  & Correct & Log Likelihood & $v_1$ & $v_2$& $v_3$ & $v_4$ & $v_5$   \\ 
            $R/\varepsilon$ & structure & (Equation \ref{Eq:cost}) & $\begin{pmatrix} 0\\0\\0\end{pmatrix}$ & $\begin{pmatrix}4.6\\ 6.24\\ 0\end{pmatrix}$ & $\begin{pmatrix}4.86\\ 0.51\\ 3.47\end{pmatrix}$ & $\begin{pmatrix} -1.32\\ 6.29\\ 4\end{pmatrix}$ &  $\begin{pmatrix}-4.23\\ -3.48\\ -3\end{pmatrix}$\\ \hline
            $4$     & No    & - & - & - & - & - & -     \\ \hline
            $6$     & Yes   & $-33.183$  & $\begin{pmatrix} 0.00 \\ 0.03 \\ 0.01\end{pmatrix}$ & $\begin{pmatrix} 4.59 \\ 6.23 \\ -0.02 \end{pmatrix}$ & $\begin{pmatrix} 4.56 \\ 0.56 \\ 3.43 \end{pmatrix}$ & $\begin{pmatrix} -1.30 \\ 6.26 \\ 3.96\end{pmatrix}$ & $\begin{pmatrix} -4.24 \\ -3.46 \\ -3.02 \end{pmatrix}$      \\ \hline
            $8$     & Yes   & $-32.97$ & $\begin{pmatrix} 0.00 \\ 0.02 \\ 0.01 \end{pmatrix}$ & $\begin{pmatrix} 4.59 \\ 6.23 \\ -.02 \end{pmatrix}$ & $\begin{pmatrix} 4.86 \\ 0.56 \\ 3.42\end{pmatrix}$ & $\begin{pmatrix} -1.29 \\ 6.26 \\ 3.96\end{pmatrix}$ & $\begin{pmatrix} -4.22\\ -3.45 \\ -3.01\end{pmatrix}$\\ \hline
            $10$    & Yes   & $-33.33$ & $\begin{pmatrix}0.01 \\ 0.03 \\ 0.01 \end{pmatrix}$ &$\begin{pmatrix}4.59 \\6.22\\-0.02 \end{pmatrix}$ & $\begin{pmatrix} 4.86\\0.56\\3.42\end{pmatrix}$& $\begin{pmatrix} -1.26\\6.24\\3.95\end{pmatrix}$&$\begin{pmatrix} -4.19\\3.42\\-2.99\end{pmatrix}$\\ \hline
            $12$    & Yes   & $-33.84$ &$\begin{pmatrix} 0.01\\0.03\\0.01\end{pmatrix}$ &$\begin{pmatrix} 4.59\\6.23\\-0.02\end{pmatrix}$ & $\begin{pmatrix}4.86\\0.56\\3.42 \end{pmatrix}$ &$\begin{pmatrix} -1.26\\6.24 \\3.95 \end{pmatrix}$ &$\begin{pmatrix} -4.14\\-3.38\\-2.96 \end{pmatrix}$  \\ \hline
            $14$    & Yes   & $-36.61$ & $\begin{pmatrix} 0.01\\0.03\\0.01\end{pmatrix}$& $\begin{pmatrix} 4.59\\6.26\\-0.03\end{pmatrix}$ &$\begin{pmatrix}4.86\\0.56\\3.43 \end{pmatrix}$ & $\begin{pmatrix}-1.26\\6.23\\3.95 \end{pmatrix}$& $\begin{pmatrix}-4.00\\-3.27\\-2.56 \end{pmatrix}$\\ \hline
            $16$    & Yes   & $-45.30$ &$\begin{pmatrix} 0.02\\0.03\\0.01 \end{pmatrix}$ &$\begin{pmatrix} 4.58\\6.27\\-0.05 \end{pmatrix}$ & $\begin{pmatrix}4.56\\0.56\\3.43 \end{pmatrix}$ & $\begin{pmatrix} -0.70\\ 3.70\\ 2.33 \end{pmatrix}$& $\begin{pmatrix}-3.96\\-3.22\\-2.81 \end{pmatrix}$\\ \hline
        \end{tabular}
        \caption{Summary of output of the algorithm for various ratios $\frac{R}{\varepsilon}$. Recall we wish to maximise Equation \ref{Eq:cost}. The last 5 columns are the vertex locations obtained.}\label{tab:sims}

    \end{table}
    
    \begin{figure}[h]
    	\centering
        \begin{subfigure}[b]{0.49\textwidth}
	        \includegraphics[width=\textwidth]{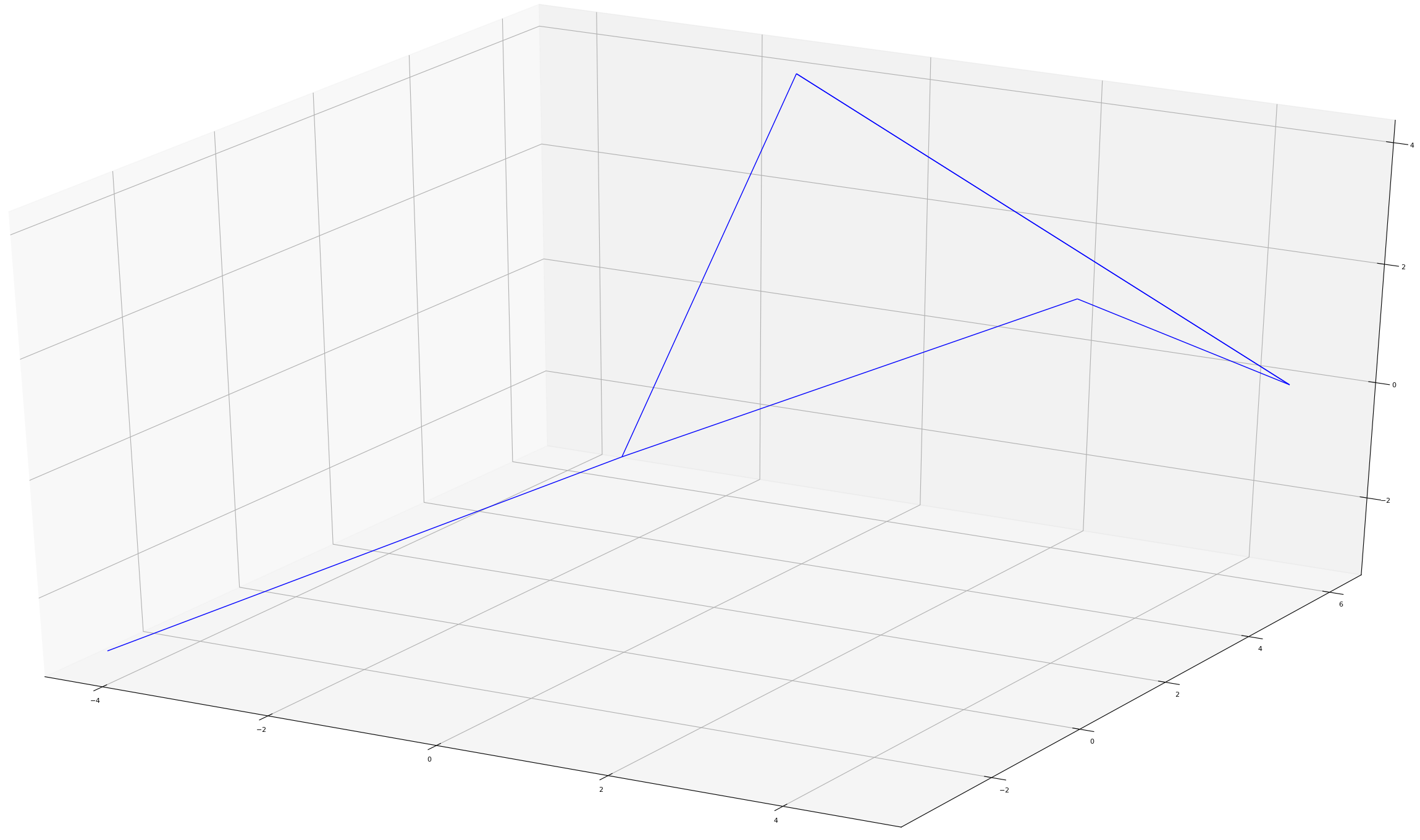}
	        \caption{Embedded graph $|G|$.}\label{fig:graph}
	    \end{subfigure}
	    \hfill 
        \begin{subfigure}[b]{0.49\textwidth}
    	    \includegraphics[width=\textwidth]{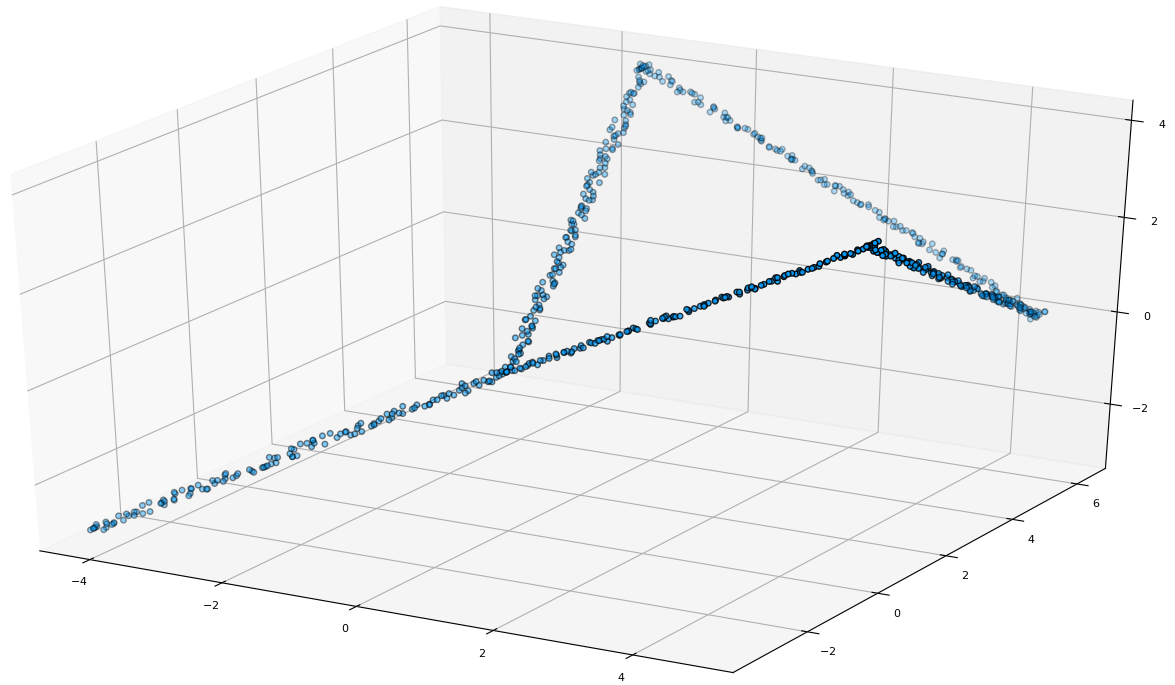}
    	    \caption{$\varepsilon$-sample $P$.}\label{fig:sample}
        \end{subfigure}
        \\
        \begin{subfigure}[b]{0.49\textwidth}
	        \includegraphics[width=\textwidth]{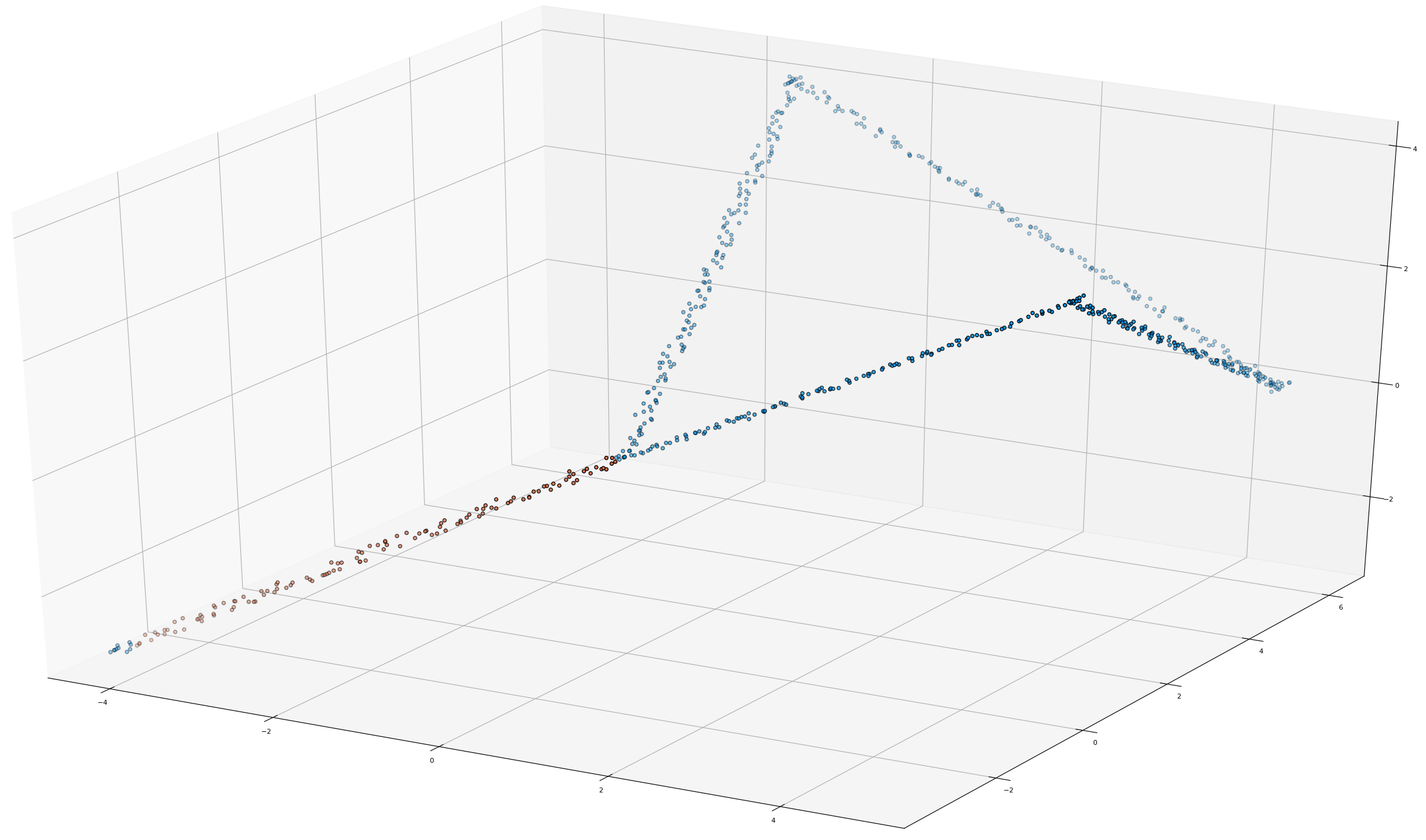}
	        \caption{$\frac{R}{\varepsilon}=4$: 2 vertex and 1 edge cluster.}\label{fig:par4}
        \end{subfigure}
        \begin{subfigure}[b]{0.49\textwidth}
	        \includegraphics[width=\textwidth]{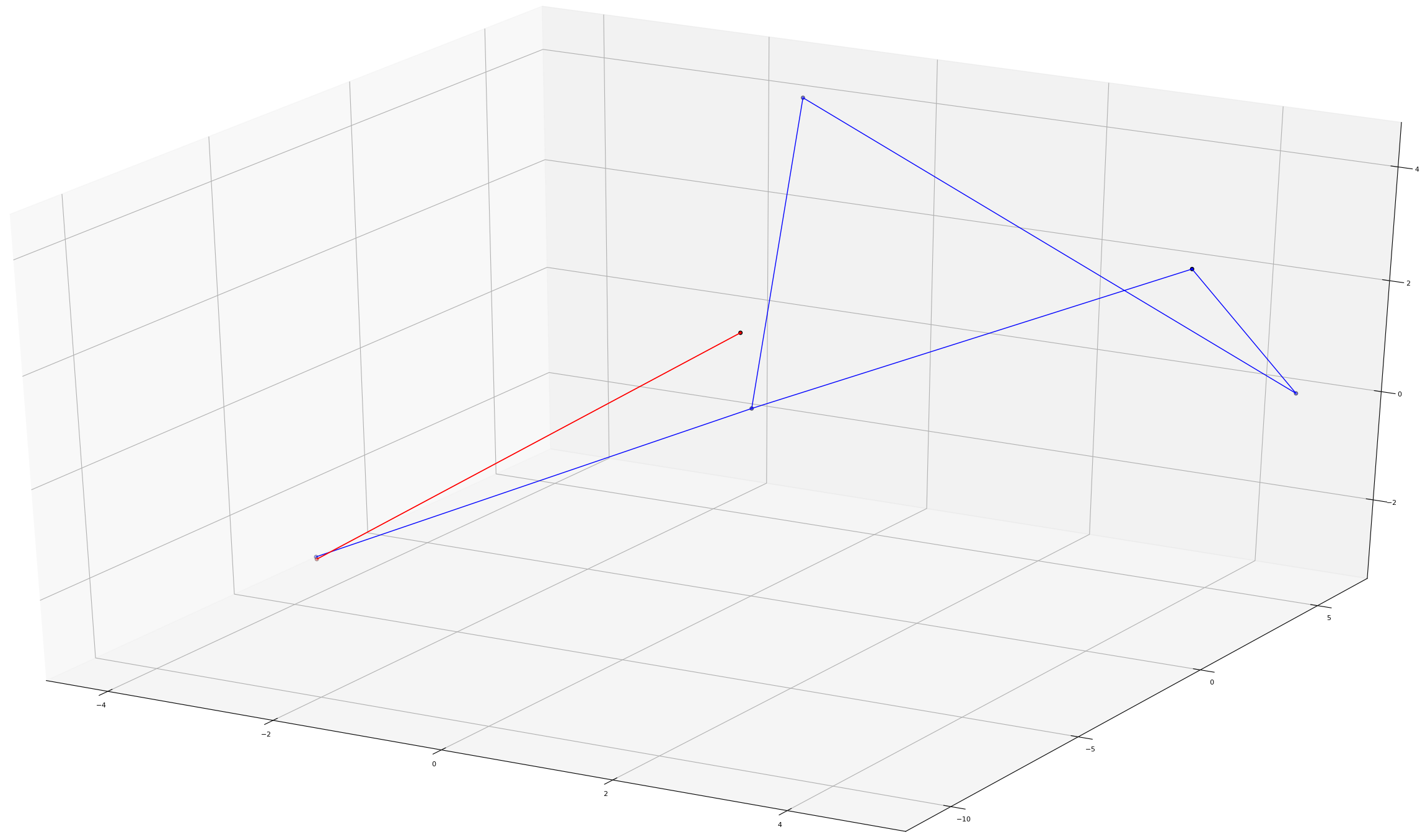}
    	    \caption{Model using $\frac{R}{\varepsilon}=4$ in red.}\label{fig:mod4}
        \end{subfigure}
        \\
        \begin{subfigure}[b]{0.49\textwidth}
    	    \includegraphics[width=\textwidth]{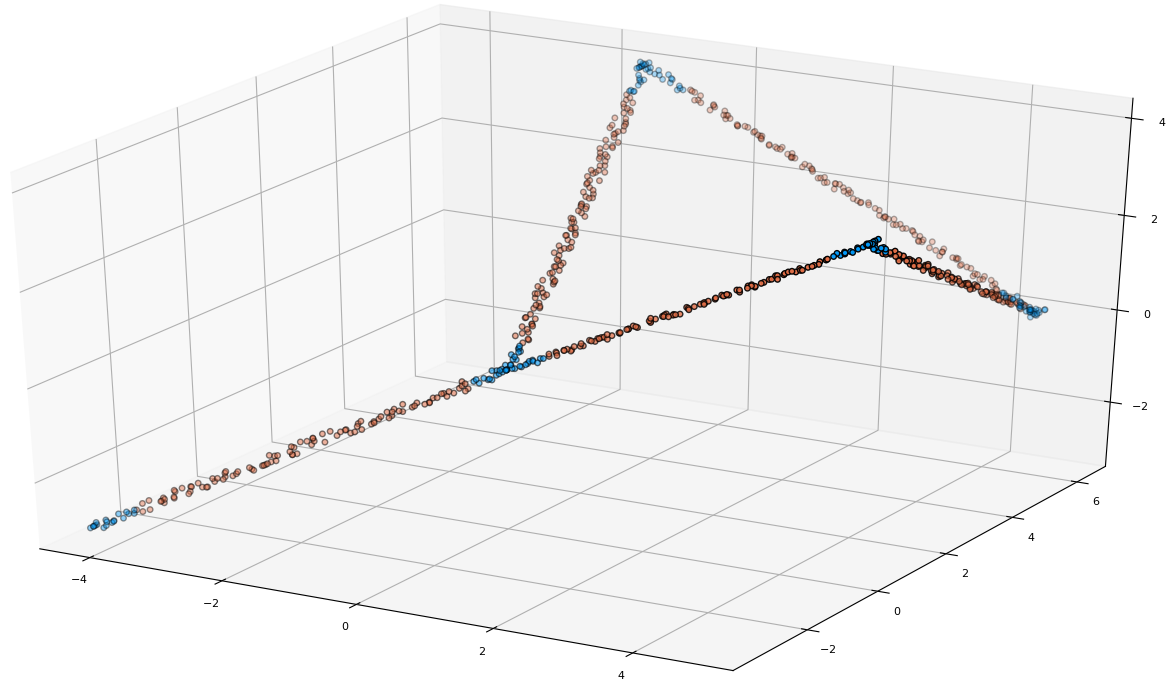}
    	    \caption{$\frac{R}{\varepsilon}=8$: 5 vertex and 5 edge clusters.}\label{fig:par8}
        \end{subfigure}
        \hfill
        \begin{subfigure}[b]{0.49\textwidth}
    	    \includegraphics[width=\textwidth]{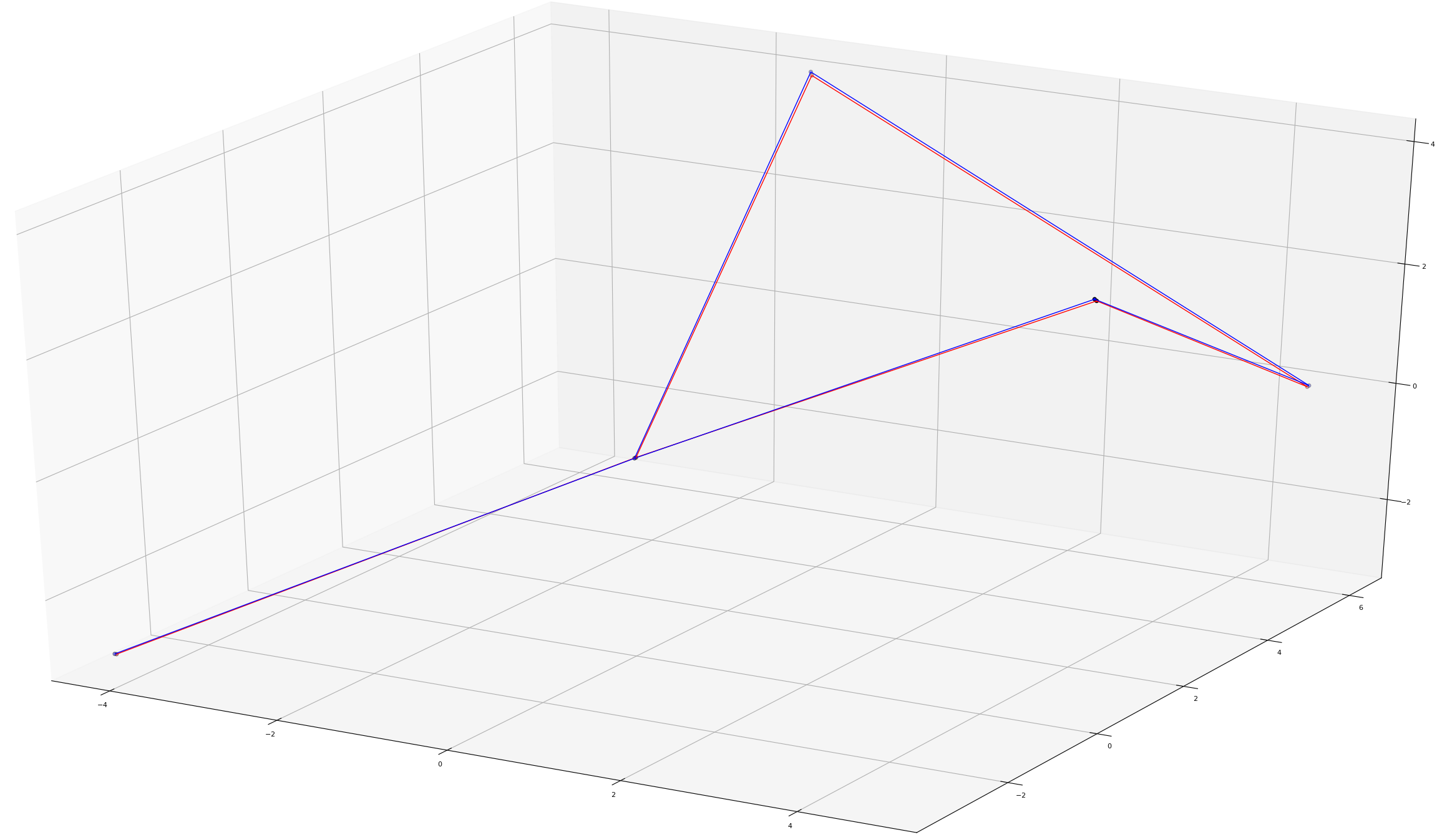}
    	    \caption{Model using $\frac{R}{\varepsilon}=8$ in red.}\label{fig:mod8}
        \end{subfigure}
        \\
        \begin{subfigure}[b]{0.49\textwidth}
    	    \includegraphics[width=\textwidth]{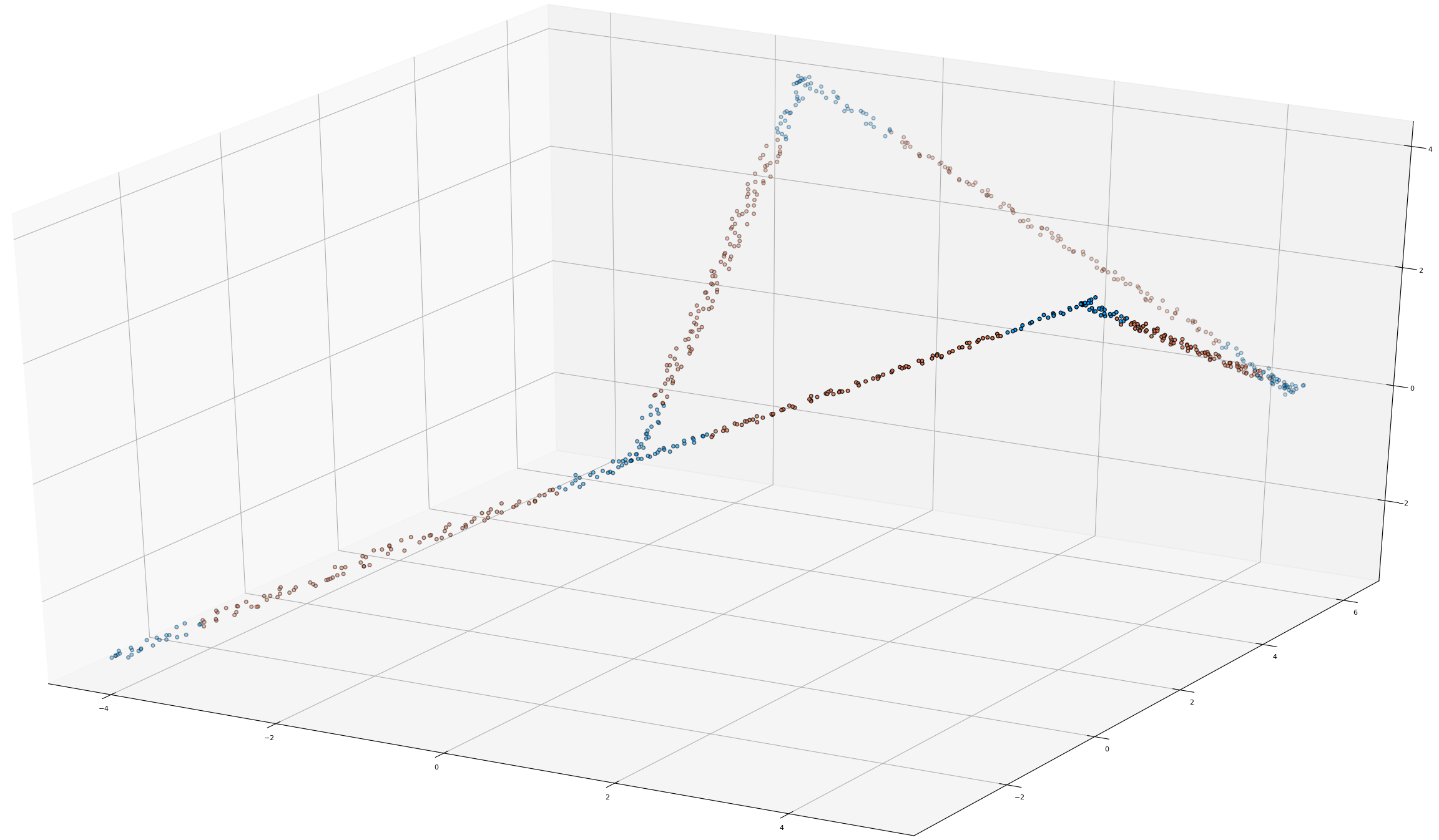}
    	    \caption{$\frac{R}{\varepsilon}=12$: 5 vertex and 5 edge clusters.}\label{fig:par12}
        \end{subfigure}
        \hfill
        \begin{subfigure}[b]{0.49\textwidth}
    	    \includegraphics[width=\textwidth]{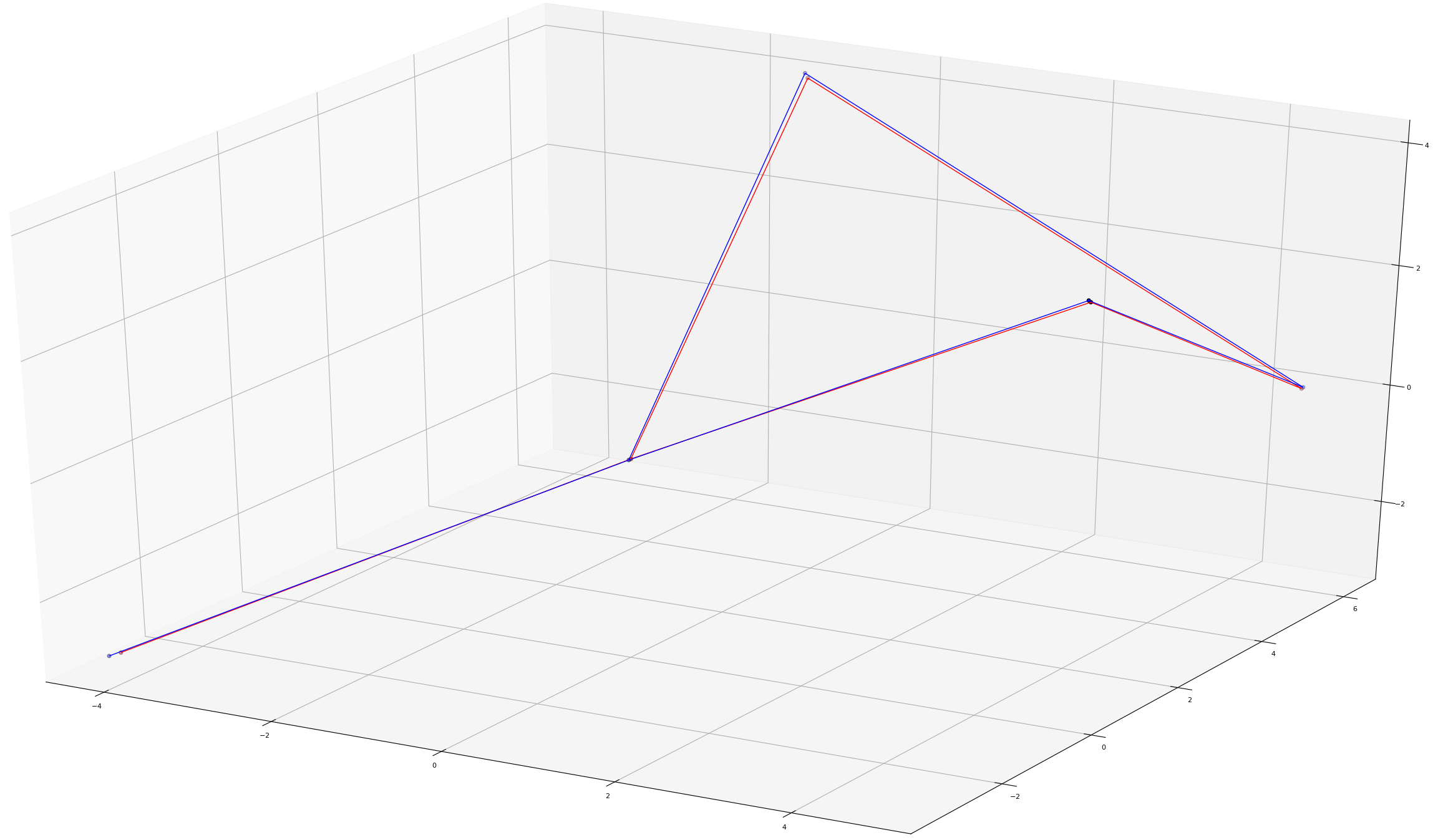}
    	    \caption{ Model using $\frac{R}{\varepsilon}=12$ in red.}\label{fig:mod12}
        \end{subfigure}
        \caption{}\label{fig:collected}    
    \end{figure}
    
    Comparing the log likelihood of the models obtained using $\frac{R}{\varepsilon}=8$ ($-2.3712314714356437$) and $\frac{R}{\varepsilon}=12$ ($-2.783827546761547$), it is clear that while we have shown that $R \geq 12 \varepsilon$ is sufficient to prove correctness of the algorithm, smaller ratios can also identify an isomorphic graph structure, and result in a  higher log likelihood model. In practice, this suggest that we can improve the process by first using $R \geq 12 \varepsilon$ to obtain the correct structure, and then decreasing the ratio to model the graph, stopping when we still obtain the correct graph structure and maximise the log likelihood. 

\section{Future Directions}\label{sec:future}

    The algorithm presented in this paper focuses on recovering and modelling an embedded graph $(G, \phi_G)$ given an $\varepsilon$-sample $P$. Stratified spaces, however, are not restricted to consisting of $0$- and $1$-dimensional pieces, nor are they restricted to being simplicial complexes. We can consider emebeddings of CW complexes, where a stratum is embedded as a semi-algebraic set. 
   
    Unfortunately, the algorithm in this paper does not naively extend to higher simplicial complexes or CW complexes. In particular, to recover embedded CW complexes, we need to remove the assumption that strata are embedded as convex hulls (linearity). Hence, there are two distinct paths forward:
    
    \begin{enumerate}
   		 \item develop an algorithm which identifies the abstract structure of simplicial complexes with $2$-dimensional simplices,
   		 \item explore methods for removing the linearity assumption (even for graphs).
  	\end{enumerate}

    Focusing on increasing the dimension of the cells in the simplicial complex, the next step is to allow 2-simplicies and partition an $\varepsilon$-sample $P$ into three parts $P_0, P_1,$ and $P_2$. One approach is a peeling argument: first we determine the points in $P_2$, and then apply the current algorithm to $P\setminus P_2$ to obtain $P_1$ and $P_0$. Complications with this include ensuring that points are not over-assigned to $P\setminus P_2$, as this can result in $P\setminus P_2$ not being suitable as input for the current algorithm. To appropriatly partition $P$, we hope to exploit the relationship between $(R, \varepsilon)$-local structure and local homology. For graphs, we saw that the dimension $1$ local homology at a point $x$ contains topological information, which corresponds to the number of points in the intersection of the $|G|$ with a ball of small radius $r$ around $x$, and if there are 2 points, their relative geometry providing more information. By generalising the $(R,\varepsilon)$-local structure appropriately, we hope to see a correspondence with information contained in higher homology groups, and augment this with other geometrical information. 
    
    To remove the linearity assumption, we need to address a long standing problem in computational algebraic geometry: learning algebraic varieties from noisy samples. In \cite{learnvars}, Breiding et. al develop an algorithm which is robust to machine error, but not sampling noise. The algorithm has also been found to fail when given large data sets sampled from simple varieties. These issues need to be overcome before we can remove the linearity assumption.

\section*{Acknowledgments} 
    The first and third authors are supported by an Australian Government Research Training Program (RTP) Fee Offset Scholarship. The authors would also like to thank Jane Tan and Martin Helmer for detailed comments on early drafts, and Cale Rankin and Ivo Vekemans for insightful discussions.


\medskip
Received xxxx 20xx; revised xxxx 20xx.
\medskip

\end{document}